\documentclass[journal]{IEEEtran}
\ifCLASSINFOpdf
\usepackage[pdftex]{graphicx}
\else
\fi
\usepackage{subfigure}
\usepackage{amsmath, amssymb, amsthm}
\usepackage{amsfonts}
\usepackage{algorithm}
\usepackage{algpseudocode}
\ifCLASSOPTIONcompsoc
\usepackage[caption=false,font=normalsize,labelfont=sf,textfont=sf]{subfig}
\else
\usepackage[caption=false,font=footnotesize]{subfig}
\fi

\usepackage{color}
\usepackage{xcolor}
\usepackage{tabularx}
\usepackage{multirow}
\usepackage{tablefootnote}

\newtheorem{assumption}{Assumption}
\newtheorem{lemma}{Lemma}
\newtheorem{remark}{Remark}
\newtheorem{theorem}{Theorem}

\newcommand{\va}{{\mathbf{a}}}
\newcommand{\vb}{{\mathbf{b}}}

\newcommand{\vh}{{\mathbf{h}}}

\newcommand{\vu}{{\mathbf{u}}}
\newcommand{\vv}{{\mathbf{v}}}
\newcommand{\vw}{{\mathbf{w}}}
\newcommand{\vx}{{\mathbf{x}}}
\newcommand{\vy}{{\mathbf{y}}}

\newcommand{\vA}{{\mathbf{A}}}

\newcommand{\vC}{{\mathbf{C}}}

\newcommand{\vF}{{\mathbf{F}}}

\newcommand{\vH}{{\mathbf{H}}}
\newcommand{\vI}{{\mathbf{I}}}

\newcommand{\vQ}{{\mathbf{Q}}}
\newcommand{\vR}{{\mathbf{R}}}

\newcommand{\vU}{{\mathbf{U}}}
\newcommand{\vV}{{\mathbf{V}}}

\newcommand{\vX}{{\mathbf{X}}}
\newcommand{\vY}{{\mathbf{Y}}}

\newcommand{\cC}{{\mathcal{C}}}

\newcommand{\cF}{{\mathcal{F}}}

\newcommand{\cQ}{{\mathcal{Q}}}

\newcommand{\EE}{{\mathbb{E}}}
\newcommand{\RR}{\mathbb{R}}
\newcommand{\vzero}{\mathbf{0}}
\newcommand{\vone}{{\mathbf{1}}}
\newcommand{\oX}{\overline{\vX}}
\newcommand{\oY}{\overline{\vY}}
\newcommand{\ox}{\overline{\vx}}

\newcommand{\T}{\intercal}

\newcommand{\tX}{\widetilde{\vX}}
\newcommand{\tY}{\widetilde{\vY}}

\newcommand{\hlambda}{\widehat{\lambda}}
\newcommand{\vzeta}{\boldsymbol{\zeta}}

\newcommand{\norm}[1]{\left\| #1 \right\|}      % norm 

\begin{document}
	\title{A Linearly Convergent Robust Compressed Push-Pull Method for Decentralized Optimization}
	
	%\author{Yiwei Liao, Zhuorui Li, and Shi Pu% <-this % stops a space
		%\thanks{ Parts of the results appear in Proceedings of the 57th IEEE Conference on Decision and Control. (Corresponding authors: .)}
		%\thanks{This work was supported in parts by the NSF grant .}
		%\thanks{X.~Y1 is with .  (e-mail: XY1@hh.edu.cn).}% <-this % stops a space

 	\author{Yiwei Liao, Zhuorui Li,  and Shi Pu% <-this % stops a space
		\thanks{Yiwei Liao is with the School of Data Science, The Chinese University of Hong Kong, Shenzhen, China and also with the Shcool of Information Science and Technology, University of Science and Technology of China, Hefei, China. 
		Zhuorui Li is with the H. Milton Stewart School of Industrial and System Engineering, Georgia Institute of Technology, Atlanta, USA.  
		Shi Pu is with the School of Data Science, The Chinese University of Hong Kong, Shenzhen, China.
		{\tt\small (emails: lyw@alu.scu.edu.cn, lizhuorui27@gmail.com,  pushi@cuhk.edu.cn)}}%
	}

	\maketitle
	
\begin{abstract}
In the modern paradigm of multi-agent networks, communication has become one of the main bottlenecks for decentralized optimization, where a large number of agents are involved in minimizing the average of the local cost functions. In this paper, we propose a robust compressed push-pull algorithm (RCPP) that combines gradient tracking with communication compression. In particular, RCPP is compatible with a much more general class of compression operators that allow both relative and absolute compression errors. We show that RCPP achieves linear convergence rate for smooth objective functions satisfying the Polyak-Łojasiewicz condition over general directed networks. Numerical examples verify the theoretical findings and demonstrate the efficiency, flexibility, and robustness of the proposed algorithm.
\end{abstract}
	
\begin{IEEEkeywords}
Decentralized optimization, robust communication compression, directed graph,  gradient tracking, linear convergence.
\end{IEEEkeywords}

	% \fi
	\IEEEpeerreviewmaketitle

\section{Introduction}
	\IEEEPARstart{I}{n} this paper, we study the decentralized optimization problem: 
	\begin{equation} \label{problem}
		\min _{\vx\in\RR^p}~  f(\vx):=\frac{1}{n}\sum _{i=1}^n f_i(\vx),
	\end{equation}
	where  $n$ is the number of agents, $\vx$ is the global decision variable, and each agent $i$ only has access to its local objective function $f_i: \RR^p\rightarrow \RR$. The goal is to find an optimal and consensual solution through local computation and local sharing of information in a directed communication network. 
	
	Decentralized algorithms for solving \eqref{problem} were well studied in recent years. The seminal work \cite{Nedic2009distributed} proposed the distributed subgradient descent (DGD) method, where each agent updates its local copy by mixing with the received copies from neighbors in the network and moving towards the local gradient descent direction. However, under a constant step-size, DGD only converges to a neighborhood of the optimal solution. To obtain better convergence results, various works with bias-correction techniques were proposed, including EXTRA \cite{Shi2015Extra}, exact diffusion \cite{yuan2019exact1}, and gradient tracking based methods \cite{Xu2015Augmented,Di2016Next,Nedic2017achieving,Qu2018Harnessing}. These methods achieve linear convergence for minimizing strongly convex and smooth objective functions. Under the more general directed network topology, several modifications have been considered; see \cite{Tsianos2012PushSum,Nedic2015Distributed,Nedic2017achieving,zeng2017extrapush,xi2017dextra,xi2017add,Xin2018Linear,Xin2020General,Pu2020Robust,Pu2021Push,sun2022distributed} and the references therein. 
    The recent papers \cite{Nedic2018Network,yang2019survey} provided a comprehensive survey on decentralized algorithms.
	 
    In decentralized computation, exchanging complete information between neighboring agents may suffer from the communication bottleneck due to the limited energy and/or bandwidth. One of the promising means for reducing the communication costs is applying compression operators \cite{tang2018communication, Koloskova2019ChocoSGD, koloskova*2020decentralized,liu2020linear,Kajiyama2020Linear,kovalev2021linearly,Song2022CPP,Liao2022CGT,xiong2022quantized}. Most of the works have considered the relative compression error assumption, including unbiased compressors \cite{liu2020linear,kovalev2021linearly,Song2022CPP} and contractive biased compressors \cite{Koloskova2019ChocoSGD, koloskova*2020decentralized}, or the unification of them \cite{Liao2022CGT}. Recently, a few works have also considered quantized compression operators with absolute compression errors \cite{xiong2022quantized,Kajiyama2020Linear,Magnusson2020Maintaining}. 
    To explore a unified framework for both relative and absolute compression errors, the work in \cite{michelusi2022finitebit} studied finite-bit quantization, but the absolute compression error needs to diminish exponentially fast for the desired convergence. In \cite{nassif2022quantization}, the unbiased relative compression error was considered together with the absolute compression error, but the latter slows down the algorithmic convergence. 
	 
    In this paper, we propose a robust compressed push-pull method (RCPP) for decentralized optimization with communication compression over general directed networks. In particular, we consider a more general assumption on the communication compressors, which unifies both relative and absolute compression errors. By employing the dynamic scaling compression technique, RCPP provably achieves linear convergence for minimizing smooth objective functions satisfying the Polyak-Łojasiewicz inequality (PL condition) under the general class of compression operators. 
  
    The main contribution of this paper is summarized as follows:
\begin{itemize}
\item For decentralized optimization with communication compression, we consider a general class of compression operators, which unifies the commonly used relative and absolute error compression assumptions. Such a condition is most general in the decentralized optimization literature to the best of our knowledge.
 \item We propose a new method called the robust compressed push-pull algorithm that works over general directed networks. Based on the dynamic scaling compression technique, RCPP provably achieves linear convergence for minimizing smooth objective functions satisfying the PL condition under the general unified assumption on the compression operators.
\item Numerical results demonstrate that RCPP is efficient compared to the state-of-the-art methods and robust under various compressors.
\end{itemize}
In Table~\ref{table:Categorization},  we compare this paper with related works regarding the assumptions on the compression operators, objective functions, graph topologies and convergence guarantees. 

    The rest of this paper is organized as follows. We introduce the notation in Section \ref{subsec: Notation}.
    In  Section~\ref{Sec: Preliminary}, we state the standing assumptions and discuss the compression methods. In Section~\ref{Sec: RCPP}, we introduce the RCPP method. In Section~\ref{Sec: CA}, we establish the linear convergence of RCPP under communication compression. Numerical experiments are provided to verify the theoretical findings in Section~\ref{Sec: Simulation}. Finally, conclusions are given in Section~\ref{Sec: Conclusion}.
	\vspace{-3mm}
\begin{table}[H]
	\centering
	\caption{\small Comparison of related works on decentralized optimization with communication compression.}	
	\label{table:Categorization}
	\vspace{-1mm}
	%\scalebox{1}{
		\begin{tabular}{@{}c@{}c@{~~}c@{}c@{}c@{~}c@{}}
			\hline
			References       &relative    & absolute & convergence  & graph & function\\
			\hline
			\cite{Koloskova2019ChocoSGD,koloskova*2020decentralized}   &C $^1$   &$\times$    &sublinear &Und &SVX $^4$\\
			\cite{liu2020linear,kovalev2021linearly}  &U      &$\times$      &linear &Und &SVX\\
			\cite{Zhang2023Innovation,yau2022docom}     &C     &$\times$      &linear &Und &SVX\cite{Zhang2023Innovation}, PL\cite{yau2022docom}\\
			\cite{Liao2022CGT,Yi2022Communicationb} &G  & $\times$ & linear &Und &SVX\cite{Liao2022CGT}, PL\cite{Yi2022Communicationb}\\
			\cite{Kajiyama2020Linear,Magnusson2020Maintaining,Yi2022Communicationb} &$\times$   &Q    &linear &Und &SVX\cite{Kajiyama2020Linear,Magnusson2020Maintaining}, PL\cite{Yi2022Communicationb}\\
			\cite{michelusi2022finitebit}  &C   &dim-d    &linear* $^2$ &Und &SVX\\
			\cite{nassif2022quantization}   &U  &$\checkmark$  &neighborhood $^3$  &Und &SVX\\	
			\cite{Song2022CPP}  &U       &$\times$         &linear &Di &SVX\\
			\cite{xiong2022quantized} &$\times$            &Q    &linear &Di &SVX\\
			\textbf{our paper}       &\textbf{G}         &$\pmb{\checkmark}$        &\textbf{linear}   &\textbf{Di}     &\textbf{PL}\\	
			\hline
		\end{tabular}
	%}
\end{table}
\vspace{-3mm}
\footnotesize{$^1$ `C', `U', `G' represent contractive biased, unbiased, general relative compression assumptions, respectively. `dim-d' and `Q' represent dimension-dependent absolute compression assumption and quantizer, respectively. `Und' and `Di' denote undirected and directed graphs, respectively.}

\footnotesize{$^2$ * The result has extra requirement, e.g., exponentially decaying error.}

\footnotesize{$^3$ The algorithm converges to the neighborhood of the optimal solution.}

\footnotesize{$^4$ `SVX' and `PL' represent strongly convex functions and the PL condition, respectively.}
\normalsize

\subsection{Notation}\label{subsec: Notation}
	A vector is viewed as a column by default. The $n$-dimensional column vector with all entries equal to $1$ is denoted by $\mathbf{1}$. Each agent $i$ holds a local copy $\vx_i\in\mathbb{R}^p$ of the decision variable and an auxiliary variable $\vy_i\in\mathbb{R}^p$ to track the average gradient. Vectors $\vx_{i}^{k}$ and $\vy_{i}^{k}$ represent their corresponding values at the $k$-th iteration. For simplicity, denote the aggregated variables as
	$\vX := [\vx_1, \vx_2, \ldots, \vx_n]^{\T}\in\mathbb{R}^{n\times p}$, 
	$ \vY := [\vy_1, \vy_2, \ldots, \vy_n]^{\T}\in\mathbb{R}^{n\times p}$. 
	At step $k$, $\vX^{k}$ and $\vY^{k}$ represent their corresponding values. The other aggregated variables $\vH_{x}$, $\vH_{y}$, $\vQ_{x}$, $\vQ_{y}$, $\widehat{\vX}$, $\widehat{\vY}$, $\tX$, and $\tY$ are defined similarly. 
	The aggregated gradients are 
	$
	\nabla \vF(\vX):=\left[\nabla f_1(\vx_1), \nabla f_2(\vx_2), \ldots, \nabla f_n(\vx_n)\right]^{\T}\in\mathbb{R}^{n\times p}. 
	$  
	With slight notation abuse, the gradients $\nabla f_i(\vx_i)$ and $\nabla f(\vx)$ are occasionally regarded as row vectors, and the average of all the local gradients is
	$
	{\nabla} \overline {\vF}(\vX) := \frac{1}{n}{\vone^\T \nabla \vF(\vX)}=\frac{1}{n}\sum_{i=1}^n\nabla f_i(\vx_i).
	$
    The notations $\|\cdot\|$ and $\|\cdot\|_F$ define the Euclidean norm of a vector and the Frobenius norm of a matrix, respectively. %Denote $\rho(\vM)$ as the spectral radius of a square matrix $\vM$.  
    
    The set of nodes (agents) is denoted by $\mathcal{N}=\{1,2,\ldots,n\}$.  A directed graph (digraph) is a pair $\mathcal{G}=(\mathcal{N},\mathcal{E})$, where the edge set $\mathcal{E}\subseteq \mathcal{N}\times \mathcal{N}$ consists of ordered pairs of nodes. If there exists a directed edge from node $i$ to node $j$ in $\mathcal{G}$, or $(i,j)\in\mathcal{E}$, then $i$ is called the parent node, and $j$ is the child node. The parent node can directly transmit information to the child node, but not the other way around. 
    % The in-neighbor set and out-neighbor set of node $i$ are given by $\mathcal{N}_{i}^{\text{in}}=\{j:(j,i)\in\mathcal{E}\}\cup \{i\}$ and $\mathcal{N}_{i}^{\text{out}}=\{j:(i,j)\in\mathcal{E}\}\cup \{i\}$, respectively. 
    Let $\mathcal{G}_\mathbf{B}=(\mathcal{N},\mathcal{E}_\mathbf{B})$ denote a digraph induced by a nonnegative square matrix $\mathbf{B}$, where $(i,j)\in\mathcal{E}_\mathbf{B}$  if and only if $\mathbf{B}_{ji}>0$. In addition, $\mathcal{R}_\mathbf{B}$ is the set of roots of all the possible spanning trees in $\mathcal{G}_\mathbf{B}$.
  
\section{Problem Formulation}\label{Sec: Preliminary}
    In this section, we first provide the basic assumptions on the communication graphs and the objective functions. Then, we introduce a general assumption on the compression operators to unify both the relative and absolute compression errors.
\subsection{Communication graphs and objective functions}
    Consider the following conditions on the communication graphs among the agents and the corresponding mixing matrices.
\begin{assumption}\label{Assumption: network}
	The matrices $\vR$ and $\vC$ are both supported by  a strongly connected graph $\mathcal{G} = (\mathcal{N},\mathcal{E})$, i.e., 
	$\mathcal{E}_{\vR} = \{(j,i)\in\mathcal{N}\times\mathcal{N} \big| \vR_{ij} > 0 \} \subset \mathcal{E}$ 
	and 
	$\mathcal{E}_{\vC} = \{(j,i)\in\mathcal{N}\times\mathcal{N} \big| \vC_{ij} > 0 \} \subset \mathcal{E}$. 
	The matrix $\vR$ is row stochastic, and $\vC$ is column stochastic, i.e., $\vR\vone=\vone$ and $\vone^\T\vC=\vone^\T$. In addition, $\mathcal{R}_{\vR} \cap \mathcal{R}_{\vC^\T} \neq \emptyset$.
\end{assumption}
\begin{remark}
	Assumption~\ref{Assumption: network} is weaker than requiring both $\mathcal{G}_{\vR}$ and $\mathcal{G}_{\vC}$ are strongly connected \cite{Pu2021Push}. It implies that $\vR$ has a unique nonnegative left eigenvector $\vu_{R}$ w.r.t. eigenvalue $1$ with $\vu_{R}^\T\vone=n$, and $\vC$ has a unique nonnegative right eigenvector $\vu_{C}$ w.r.t. eigenvalue $1$ such that $\vu_{C}^\T\vone=n$. The nonzero entries of $\vu_{R}$ and $\vu_{C}$ correspond to the nodes in $\mathcal{R}_{\vR}$ and $\mathcal{R}_{\vC^\T}$, respectively. Since $\mathcal{R}_{\vR} \cap \mathcal{R}_{\vC^\T} \neq \emptyset$, we have $\vu_{R}^\T \vu_{C} > 0$.	
\end{remark}
    The objective functions are assumed to satisfy the following condition.
\begin{assumption}\label{Assumption: function}
	The objective function $f$ satisfies the Polyak-Łojasiewicz inequality (PL condition), i.e.,
	\begin{equation}\label{PL}
	\norm{\nabla f(\vx)}^2\geq 2 \mu (f(\vx)-f(\vx^*)),
	\end{equation}
	where   $\vx^*$  is an optimal  solution to problem \eqref{problem}. 
	For each agent $i$,  its gradient is $L_i$-Lipschitz continuous, i.e.,
	\begin{equation}\label{L-smooth}
	%& \langle \nabla f_i(\vx)-\nabla f_i(\vx'),\vx-\vx'\rangle\ge \mu_i\|\vx-\vx'\|^2,\\
	\|\nabla f_i(\vx)-\nabla f_i(\vx')\|\le L_i \|\vx-\vx'\|,\; \forall\vx,\vx'\in\mathbb{R}^p.
	\end{equation}
\end{assumption}
\begin{remark}
    If $f$ is $\mu$-strongly convex as commonly assumed, the PL condition is automatically satisfied. From Assumption \ref{Assumption: function}, the gradient of $f$ is $L$-Lipschitz continuous, where $L=\max{\{L_i\}}$. We denote $\kappa=L/\mu$ as the condition number.
\end{remark}	
 
\subsection{A unified compression assumption}\label{Sec: Compression}
We now present a general assumption on the compression operators which incorporates both relative and absolute compression errors. 	
\begin{assumption}\label{Assumption:General}
	The  compression operator $\cC \colon \RR^d \to \RR^d$  satisfies  
	\begin{align}\label{def:Ncompressor4}
		\EE_{\cC} \norm{\cC(\vx) -\vx}^2 &\leq C\norm{\vx}^2+\sigma^2,& ~\forall \vx \in \RR^d, 
	\end{align}
	for some constants $C,\sigma^2\ge 0$, and the $r$-scaling of $\cC$ satisfies 
	\begin{align}\label{def:contract4}
		\EE_{\cC} \norm{\cC(\vx)/r -\vx}^2 &\leq (1-\delta)\norm{\vx}^2+\sigma^2_{r},& ~\forall \vx \in \RR^d , 
	\end{align}
	for some constants $r>0$, $\delta\in(0,1]$ and $\sigma^2_r\ge 0$.
\end{assumption}
	Among the compression conditions considered for decentralized optimization algorithms with convergence guarantees, Assumption \ref{Assumption:General} is the weakest to the best of our knowledge. Specifically, if there is no absolute error, i.e., $\sigma^2=\sigma_{r}^2=0$, then Assumption \ref{Assumption:General} degenerates to the assumption in \cite{Liao2022CGT} that unifies the compression operators with relative errors. If there is no relative error, i.e., $C=0$ and $\delta=1$, then the condition becomes the assumption on the quantizers in \cite{Kajiyama2020Linear,xiong2022quantized}. Therefore, Assumption \ref{Assumption:General} provides a unified treatment for both relative and absolute compression errors. In addition, if $C<1$, Assumption \ref{Assumption:General} reduces to the condition in \cite{michelusi2022finitebit}. 

\section{A Robust Compressed Push-Pull Method}\label{Sec: RCPP}
    In this section, we first introduce the dynamic scaling compression technique that deals with the absolute compression error. Then, we propose the RCPP algorithm and discuss its connections with the existing methods.
    
\subsection{The dynamic scaling compression technique}
    While Assumption~\ref{Assumption:General} provides a unified condition on the compression operators, new challenges are brought to the algorithm design and analysis. Without proper treatment for the compression errors, the algorithmic performance could deteriorate, particularly due to the absolute error that may lead to compression error accumulation. To tackle the challenge, we consider the dynamic scaling compression technique \cite{Kajiyama2020Linear}.
    Consider the operator $\cQ(\vx)=s_k\cC(\vx/s_k)$, where $s_k$ is a dynamic parameter related to the iteration $k$. Then from Assumption~\ref{Assumption:General}, we have 
$\EE_{\cQ} \norm{\cQ(\vx) -\vx}^2
  = \EE_{\cC} \norm{s_k\cC(\vx/s_k) -\vx}^2
  = s_k^{2}\EE_{\cC} \norm{\cC(\vx/s_{k}) -\vx/s_{k}}^2
  \leq s_k^2(C\norm{\vx/s_k}^2+\sigma^2)
  = C\norm{\vx}^2+s_k^2\sigma^2
$.
    Similarly, we know
$
 \EE_{\cQ} \norm{\cQ(\vx)/r -\vx}^2  \leq (1-\delta)\norm{\vx}^2+s_k^2\sigma^2_{r}
$. 
   Note that only $\cC(\vx/s_k)$ needs to be transmitted during the communication process, and the recovery of signal is done by computing $\cQ(\vx)=s_k\cC(\vx/s_k)$ on the receiver's side. By using the dynamic scaling compression technique, the absolute errors can be controlled by decaying the parameter $s_k$.
   
\subsection{A robust compressed push-pull method}
   We describe the proposed RCPP method in Algorithm \ref{Alg:RCGTe}. 
   % We use $\Lambda$ to denote the diagonal matrix of the  step-sizes i.e., 
   Lines 2 and  9 represent the updates for the local decision variables and the gradient trackers, respectively. In Lines 3 and 10, the dynamic scaling compression technique is applied to execute difference compression between the local updates and the auxiliary variables. Difference compression reduces the relative compression errors \cite{Koloskova2019ChocoSGD,Liao2022CGT}, while the dynamic scaling compression controls the absolute compression errors. More specifically, the operator $\cQ$ is a dynamic scaling compressor given by $\cQ(\vx)=s_k\cC(\vx/s_k)$. 
   % and only $\cC(\vx/s_k)$ is executed in compression steps. 
   The compressed vector  $\cC((\widetilde{\vx}_{i}^{k}-\vh^{k}_{i,x})/s_k)$  is transmitted to the neighbors of agent $i$ and recovered by computing $s_k \cC((\widetilde{\vx}_{i}^{k}-\vh^{k}_{i,x})/s_k)$ after communication, where $\widetilde{\vx}_{i}^{k}$ and $\vh^{k}_{i,x}$ denote agent $i$'s local update and auxiliary variable, respectively. It is worth noting that if the dynamic scaling compression technique is not used, then the absolute compression error would accumulate and significantly impact the algorithm's convergence.
   \begin{figure}[htbp]
   	\centering
   	\begin{minipage}{.99\linewidth}
   		\begin{algorithm}[H]
   			\caption{A Robust Compressed Push-Pull Method}
   			\label{Alg:RCGTe}
   			\textbf{Input:} step-sizes $\Lambda=\text{diag}([\lambda_1,\lambda_2,\ldots,\lambda_n])$, parameters $\alpha_x,\alpha_y$, $\gamma_x,\gamma_y$, $\{s_k\}_{k\ge 0}$,  
   			initial values $\vX^{0}$, $\vY^{0}=\nabla\vF(\vX^{0})$, $\vH_{x}^0=0$, $\vH_{y}^0=0$, $\vH_{R}^0=0$, $\vH_{C}^0=0$, number of iterations $K$
   			\begin{algorithmic}[1]
   				\For{$k=0,1,2,\dots, K-1$}
   				%\Statex Dynamic scaling difference compression for the decision variables
   				\State $\widetilde{\vX}^{k}=\vX^{k}-\Lambda \vY^{k}$   %\Comment{Local decision variable update}
   				\State $\vC^{k}_{x}=\cC((\widetilde{\vX}^{k}-\vH^{k}_{x})/s_k)$   %\Comment{Dynamic scaling difference compression for the decision variables}
   				\State $\widehat{\vX}^{k}=\vH^{k}_{x}+\vQ^{k}_{x}$ $^1$ \footnotetext{$^1$ $\vQ^{k}_{x}$ is the result of  dynamic scaling compression with  $\vQ^{k}_{x}=\cQ(\widetilde{\vX}^{k}-\vH^{k}_{x})=s_k\cC((\widetilde{\vX}^{k}-\vH^{k}_{x})/s_k)=s_k\vC^{k}_{x}$. The operation for $\vQ^{k}_{y}$ is the same.}
   				\State $\widehat{\vX}_{R}^{k}=\vH^{k}_{R}+\vR\vQ^{k}_{x}$ \Comment{Communication}
   				\State $\vH^{k+1}_{x}=(1-\alpha_x)\vH^{k}_{x}+\alpha_x\widehat{\vX}^{k}$
   				\State $\vH_{R}^{k+1}=(1-\alpha_x)\vH^{k}_{R}+\alpha_x\widehat{\vX}_{R}^{k}$
   				\State $\vX^{k+1}=\widetilde{\vX}^{k}-\gamma_{x}(\widehat{\vX}^{k}-\widehat{\vX}_{R}^{k})$   %\Comment{Consensus update with communication compression}
   				%\Statex 
   				%\Statex Dynamic scaling difference compression for the gradient trackers
   				\State $\widetilde{\vY}^{k}= \vY^{k}+\nabla\vF(\vX^{k+1})-\nabla\vF(\vX^{k})$ %\Comment{Local gradient tracker update}
   				\State $\vC^{k}_{y}=\cC((\widetilde{\vY}^{k}-\vH^{k}_{y})/s_k)$ %\Comment{Dynamic scaling difference compression for the gradient trackers}
   				\State $\widehat{\vY}^{k}=\vH^{k}_{y}+\vQ^{k}_{y}$
   				\State $\widehat{\vY}_{C}^{k}=\vH^{k}_{C}+\vC\vQ^{k}_{y}$ \Comment{Communication}
   				\State $\vH^{k+1}_{y}=(1-\alpha_y)\vH^{k}_{y}+\alpha_y\widehat{\vY}^{k}$
   				\State $\vH_{C}^{k+1}=(1-\alpha_y)\vH^{k}_{C}+\alpha_y\widehat{\vY}_{C}^{k}$
   				\State $\vY^{k+1}=\widetilde{\vY}^{k}-\gamma_{y}(\widehat{\vY}^{k}-\widehat{\vY}_{C}^{k})$ %\Comment{Consensus update with communication compression}	
   				\EndFor
   				% \STATE \textbf{Output:} $\vX^{K}$
   				%\RETURN $\vX^{K}$
   			\end{algorithmic}
      \noindent\textbf{Output:} $\vX^{K},\vY^{K}$
   		\end{algorithm}
   	\end{minipage}
   \end{figure}
   
   In Lines 4 and 11, the decision variables and the gradient trackers are locally recovered, respectively. Lines 5 and 12 represent the communication steps, where each agent mixes the received compressed vectors multiplied by $s_k$.  The variables $\widehat{\vX}_{R}^{k}$ and $\widehat{\vY}_{C}^{k}$ are introduced to store the  aggregated information received from the communication updates. By introducing such  auxiliary variables, there is no need to store all the neighbors' reference points \cite{liu2020linear,Koloskova2019ChocoSGD}. Lines 6-7 and 13-14 update the auxiliary variables, where parameters $\alpha_x,\alpha_y$ control the relative compression errors; see e.g., \cite{Liao2022CGT} for reference. The consensus updates are performed in Lines 8 and 15, where $\gamma_{x},\gamma_{y}$ are the global consensus parameters to guarantee the algorithmic convergence. 

   To see the connection between RCPP and the Push-Pull/AB algorithm \cite{Xin2018Linear, Pu2021Push}, note that we have $\vH_{R}^0=\vR \vH_{x}^0$ and $\vH_{C}^0=\vC \vH_{y}^0$ from the initialization. It follows by induction that $\widehat{\vX}_{R}^{k}=\vR \widehat{\vX}^{k}$, $\widehat{\vY}_{C}^{k}=\vC \widehat{\vY}^{k}$ and  $\vH_{R}^k=\vR \vH_{x}^k$, $\vH_{C}^k=\vC \vH_{y}^k$. Recalling Lines 8 and 15 in Algorithm~\ref{Alg:RCGTe}, we have
\begin{align}
   \nonumber\vX^{k+1}=&\widetilde{\vX}^{k}-\gamma_{x}(\widehat{\vX}^{k}-\widehat{\vX}_{R}^{k})
  =\widetilde{\vX}^{k}-\gamma_{x}(\widehat{\vX}^{k}-\vR \widehat{\vX}^{k})\\
   \label{Xkplus}=&\widetilde{\vX}^{k}-\gamma_{x}(\vI-\vR)\widehat{\vX}^{k}
\end{align}
and
\begin{align}
   \nonumber \vY^{k+1}=&\widetilde{\vY}^{k}-\gamma_{y}(\widehat{\vY}^{k}-\widehat{\vY}_{R}^{k})
  =\widetilde{\vY}^{k}-\gamma_{y}(\widehat{\vY}^{k}-\vC \widehat{\vY}^{k})\\
   \label{Ykplus}=&\widetilde{\vY}^{k}-\gamma_{y}(\vI-\vC)\widehat{\vY}^{k}.
\end{align}
   If $\tX^{k}$ and $\tY^{k}$ are not compressed, i.e., $\widehat{\vX}^{k}=\tX^{k}$ and $\widehat{\vY}^{k}=\tY^{k}$, then,
   $
   \vX^{k+1}=\tX^{k}-\gamma_{x} (\vI-\vR)\tX^{k}
   =[(1-\gamma_{x})\vI+\gamma_x \vR](\vX^{k}-\Lambda \vY^{k}),
   $
   and 
   $
   \vY^{k+1}=\tY^{k}-\gamma_{y} (\vI-\vC)\tY^{k}
   =[(1-\gamma_{y})\vI+\gamma_y\vC](\vY^{k}+\nabla\vF(\vX^{k+1})-\nabla\vF(\vX^{k})).
   $
   Letting the consensus step-sizes be $\gamma_{x}=1$ and $\gamma_{y}=1$, the above updates recover those in the Push-Pull/AB algorithm \cite{Xin2018Linear, Pu2021Push}. 

   In addition, RCPP retains the property of gradient tracking based methods. From Line 15 in Algorithm 1,
 $\vone^\T \vY^{k+1}=\vone^\T  (\widetilde{\vY}^{k}-\gamma_{y}(\vI-\vC)\widehat{\vY}^{k})
 =\vone^\T ( \vY^{k}+\nabla\vF(\vX^{k+1})-\nabla\vF(\vX^{k}))
  =\vone^\T \nabla\vF(\vX^{k+1})$,
   where the second equality is from $\vone^\T (\vI-\vC)=\vzero$, and the last equality is deduced by induction given that $\vY^{0}=\nabla\vF(\vX^{0})$. Define $\oX^{k}=\frac{1}{n}\vu_{R}^{\T}\vX^{k}$	and $\oY^{k}=\frac{1}{n}\vone^{\T}\vY^{k}$.  Once $(\vx_{i}^{k})^\T\rightarrow\oX^{k}$ and $(\vy_{i}^{k})^\T\rightarrow\oY^{k}$, then each agent can track the average gradient, i.e.,  $(\vy_{i}^{k})^\T\rightarrow\oY^{k}=\frac{1}{n}\vone^\T\nabla\vF(\vX^{k})\rightarrow \frac{1}{n}\vone^\T\nabla\vF(\vone\oX^{k})$.

\section{Convergence analysis}\label{Sec: CA}
   In this section, we study the convergence property of RCPP under smooth objective functions satisfying the PL condition. 
   For simplicity of notation, denote $\Pi_{R}=\vI-\frac{\vone \vu_{R}^{\T}}{n}$, $\Pi_{C}=\vI-\frac{\vu_{C}\vone^{\T}}{n}$ and $\vX^*=(\vx^*)^\T\in\mathbb{R}^{1\times p}$. The main idea is to bound  the optimization error $\Omega_{o}^{k}:=\EE\big[f(\oX^{k})-f(\vX^*)\big]$, consensus error $\Omega_{c}^{k}:=\EE\big[\|\Pi_{R}\vX^{k}\|_R^2\big]$, gradient tracking error $\Omega_{g}^{k}:=\EE\big[\|\Pi_{C}\vY^{k}\|_C^2\big]$, and compression errors $\Omega_{cx}^{k}:=\EE\big[\| \tX^{k}- \vH^{k}_x \|_F^2\big]$ and  $\Omega_{cy}^{k}:=\EE\big[\|\vY^{k}-\vH^{k}_y\|_F^2\big]$ through a linear system of inequalities, where $\|\Pi_{C}\vY^{k}\|_R$ and $\|\Pi_{C}\vY^{k}\|_C$ are specific norms introduced in Lemma \ref{Norm_R_C}. Denote $\overline{\lambda}=\frac{1}{n}\vu_{R}^{\T}\Lambda\vu_{C}$, $\hlambda=\max\limits_{i}\{\lambda_i\}$.
% \subsection{Auxiliary results}

   We first introduce two supporting lemmas. 
\begin{lemma}\label{Norm_R_C}
	There exist invertible matrices  $\widetilde{\vR},\widetilde{\vC}$ such that the induced vector norms $\norm{\cdot}_{R}$ and $\norm{\cdot}_{C}$ satisfy $\norm{\vv}_{R}=\Vert\widetilde{\vR}\vv\Vert$ and $\norm{\vv}_{C}=\Vert\widetilde{\vC}\vv\Vert$, respectively. Then, for any $\gamma_{x},\gamma_{y} \in (0,1]$, 
	$\norm{\Pi_{R}\vR_{\gamma}}_{R}\leq 1-\theta_{R} \gamma_{x}$ 
	and 
	$\norm{\Pi_{C}\vC_{\gamma}}_{C}\leq 1-\theta_{C} \gamma_{y},$ 
	where $\vR_{\gamma}=\vI-\gamma_x(\vI-\vR)$, $\vC_{\gamma}=\vI-\gamma_y(\vI-\vC)$, $\theta_{R}$ and $\theta_{C}$ are constants in $(0,1]$. In particular, $\norm{\Pi_{R}}_{R}=\norm{\Pi_{C}}_{C}=1$.
\end{lemma}
\begin{proof}
	See the supplementary material in \cite{Song2022CPP}.
\end{proof}
\begin{lemma}\label{Lem:Yk}
	For $\norm{\vY^{k}}_{F}^{2}$, we have
	\begin{equation}\label{YkNormLem}
	\begin{aligned}
	\norm{\vY^{k}}_{F}^{2}
	\leq& 3\norm{ \Pi_{C}\vY^{k} }_{C}^{2}
	+\frac{3\norm{ \vu_{C} }^{2}}{n} L^2 \norm{ \Pi_{R}\vX^{k} }_{R}^{2} \\
	&+3 \norm{\vu_{C}}^{2}\norm{\nabla f(\oX^{k})}^{2}.
	\end{aligned}
	\end{equation}
\end{lemma}
\begin{proof}
	See Appendix \ref{Pf:Yk}.
\end{proof}
In the next lemma, we present the key linear system of inequalities. 
\begin{lemma}\label{Lem:RCPP}
	Suppose Assumptions \ref{Assumption: network}, \ref{Assumption: function} and \ref{Assumption:General} hold 
	and   $\hlambda\leq \min\big\{\frac{1}{6L},\frac{1}{6\sqrt{C}L}\big\}$. 
	Then we have% the following linear system of inequalities:
	\begin{align}\label{ineq:LMI}
	\vw^{k+1}\leq \vA \vw^{k} + \vb\norm{\vY^{k}}_{F}^{2}+\vzeta^{k},
	\end{align}
	where 
	$
	\vw^k =	\big[ 
	\Omega_{c}^{k}, 
	\Omega_{g}^{k}, 
	\Omega_{cx}^{k}, 
	\Omega_{cy}^{k}\big]^{\T}
	$ 
	and 
	$\vzeta^{k}=s_k^{2}\cdot\big[\zeta_{c},\zeta_{g},\zeta_{cx},\zeta_{cy}\big]^\T$.
\end{lemma}
\begin{proof}
	See Appendix \ref{Pf:LMI}.
\end{proof}
    The following descent lemma comes from the smoothness of the gradients and will be used for proving the main result together with Lemma \ref{Lem:RCPP}.
\begin{lemma}\label{Lem:descent}
    Suppose Assumption \ref{Assumption: function},  $\overline{\lambda}\leq\frac{1}{L}$ and $\overline{\lambda}\geq M \hlambda$ for some $M>0$ hold, we have
    	\begin{equation}\label{descent}
    	\begin{aligned}
    	&f(\oX^{k+1})\leq f(\oX^{k})-\frac{M\hlambda}{2}   \norm{\nabla f(\oX^{k})}^2 \\ %-\frac{\overline{\lambda}}{4}\norm{\frac{1}{n\overline{\lambda}}\vu_{R}^{\T}\Lambda \vY^{k}}_{2}^2 \\
    	&\qquad+E_1M\hlambda  L^2\norm{\Pi_{R}\vX^{k}}_{R}^{2}
    	+E_2M\hlambda\norm{\Pi_{C}\vY^{k}}_{C}^{2}, 
    \end{aligned}
    \end{equation}
    where  $E_1=\frac{\norm{\vu_{R}} \norm{\vu_{C}}}{n^2M}$ and $E_2=\frac{\norm{\vu_{R}}^{2}}{n^2 M^2}$. 
\end{lemma}
\begin{proof}
	See Appendix \ref{Pf:descent}.
\end{proof}

Based on the above lemmas, we demonstrate the linear convergence rate of RCPP for minimizing smooth objective functions satisfying the PL condition in the following theorem. 
\begin{theorem}\label{Thm:RCPP}
	Suppose Assumptions \ref{Assumption: network}, \ref{Assumption: function} and \ref{Assumption:General} hold, the scaling parameters $\alpha_x, \alpha_y \in (0, \frac{1}{r}]$, $\bar{\lambda}\geq M \hat{\lambda}$ for some $M>0$, 
	and the consensus step-sizes $\gamma_x,\gamma_y$ and the maximum step-size $\hlambda$ satisfy
\begin{align*}
&\hlambda \leq \Bigg\{\frac{1}{6L},\frac{1}{6\sqrt{C}L},\frac{1}{ML},
\frac{\sqrt{d_{42}}}{2\sqrt{  e_1 c_{32}c_{43}}}\frac{\theta_{C}}{L},
\\
&\min\Big\{
\frac{1}{2\sqrt{c_2}}\frac{\theta_{R}}{1-\theta_{R}\gamma_{x}},
\frac{\sqrt{c_{32}}}{\sqrt{c_3}},
\frac{\sqrt{d_{42}}}{\sqrt{c_4}},
\frac{\sqrt{d_{52}}}{\sqrt{c_5}}
\Big\}
\frac{1}{\sqrt{96E}\norm{\vu_{C}}}\frac{\gamma_{x}}{L},\\
&\min\Big\{
\frac{1}{2\sqrt{ c_2 c_{32}}}\frac{\theta_{R}}{1-\theta_{R}\gamma_{x}},
\frac{1}{ \sqrt{ c_3}},
\frac{\sqrt{d_{42}}}{\sqrt{ c_4c_{32}}},
\frac{\sqrt{d_{52}}}{ \sqrt{ c_5c_{32}}}
\Big\}
\frac{\theta_{C}\gamma_{y}}{\sqrt{48e_1}L}
\Bigg\},
\end{align*}
\begin{align*}
\gamma_{x} \leq& \Bigg\{1, 
\frac{\theta_{R}\alpha_x r \delta }{8\sqrt{ C d_{42}c_{24}}},
\frac{\sqrt{c_{32}} \alpha_x r \delta }{4\sqrt{ C d_{42}c_{34}}},
\frac{\alpha_x r \delta}{4\sqrt{C d_{44}}},
\frac{\sqrt{d_{52}} \alpha_x r \delta }{4\sqrt{C d_{42}d_{54}}},\\
&
\frac{\sqrt{M\norm{\vu_{C}}}}{2\sqrt{\norm{\vu_{R}}e_1 c_{32}}}\theta_{C}\gamma_y
\Bigg\},\\
\gamma_{y} \leq& \Bigg\{1, 
\frac{\theta_{C}e_2 d_{52}}{4 e_1 c_{32}d_{53}},
\frac{\theta_{C} (\alpha_y r \delta)^2 e_1 c_{32} }{16 C e_2 d_{52}c_{35}},
\frac{\theta_{C}(\alpha_y r \delta)^2 d_{42}}{16 C e_2 d_{52}c_{45}},\\
&
\frac{\alpha_y r \delta }{4\sqrt{C  d_{55}}}
\Bigg\},
\end{align*}
	Then, the optimization error $\Omega_{o}^{k}$ and the consensus error $\Omega_{c}^{k}$ both converge to 0 at the linear rate  $\mathcal{O}(c^{k})$, where $c\in(\tilde{\rho},1)$, where $\tilde{\rho}=\max\{1-\frac{1}{2}M\hlambda\mu,1-\frac{\theta_{R}\gamma_x}{16},1-\frac{\theta_{C}\gamma_y}{8},1-\frac{\alpha_x r \delta}{4},1-\frac{\alpha_y r \delta}{16}\}$, $s_k^2=c_0c^k$. 
\end{theorem}
\begin{proof}
	See Appendix \ref{Pf:RCPP}.
\end{proof}

\begin{remark}
	It is worth nothing that the linear convergence of RCPP does not depend on the decaying assumption of the absolute compression error as in \cite{michelusi2022finitebit}. 
\end{remark}

\section{Numerical Examples}\label{Sec: Simulation}
    In this section, we provide some numerical results to confirm the theoretical findings. Consider the following ridge regression problem,
\begin{align}\label{Ridge Regression} 
	\min_{x\in \mathbb{R}^{p}}f(x)=\frac{1}{n}\sum_{i=1}^nf_i(x)\left(=\left(u_i^{\T} x-v_i\right)^2+\rho\|x\|^2\right),
\end{align}
   where $\rho>0$ is a penalty parameter. The pair $(u_i,v_i)$ is a sample data that belongs to the $i$-th agent, where $u_i\in\mathbb{R}^p$ represents the features, and $v_i\in\mathbb{R}$ represents the observations or outputs. These parameter settings are the same as in \cite{Liao2022CGT}.

   We compare RCPP wih CPP \cite{Song2022CPP} and QDGT \cite{xiong2022quantized} for decentralized optimization over a directed graph. The row-stochastic and column-stochastic weights are randomly generated. Regarding the compressor, we consider an adaptation from the $b$-bits $\infty$-norm quantization compression method in \cite{liu2020linear}, stated below:
   \begin{align}\label{Quant}
   \cQ(\vx)=\frac{h(\|\vx\|_{\infty})} {2^{b-1}} \text{sign}(\vx)  \odot \left \lfloor{\frac{2^{b-1}|\vx|}{\|\vx\|_{\infty}}+\vu}\right \rfloor,
   \end{align}
   where $\text{sign}(\vx)$ is the sign function, $\odot$ is the Hadamard product, $|\vx|$ is the element-wise absolute value of $\vx$, and $\vu$ is a random perturbation vector uniformly distributed in $[0, 1]^{p}$. 
   
   Compared with the original compressor in \cite{liu2020linear} which computes $\|\vx\|_{\infty}$, the new operator uses the mapping $h(\|\vx\|)$ which is a random variable such that $h(\|\vx\|)=\lfloor \|\vx\| \rfloor+1$ with probability $\|\vx\|-\lfloor \|\vx\| \rfloor$ and $h(\|\vx\|)=\lfloor \|\vx\| \rfloor$ otherwise. 
   % This is because it may not be possible to use infinite-bit or full precision to represent the norm for the quantization compression method due to the machine precision limit. 
   By considering $h(\|\vx\|)$, only dynamic finite bits, i.e., about $\log_{2}(\lfloor \|\vx\| \rfloor+1)+1$ bits, need to be transmitted for nonzero norms. 
   % Thus, we further save some communication bits. 
   The quantization with the new mapping $h(\|\vx\|)$ is abbreviated as Qn, and QTn denotes the composition of quantization and Top-k compressor with the same operation. 
   Note that these compression operators produce absolute compression errors and satisfy Assumption \ref{Assumption:General}, and QTn does not satisfy the previous conditions on the compression operators. In the simulation we let $b=2$. In addition to Qn and QTn, we also consider the quantizer compression in \cite{xiong2022quantized,Kajiyama2020Linear} which satisfies the absolute compression error assumption. The quantized level is $1$, i.e., the quantized values belong to $\{-1,0,1\}$.

   In Fig. \ref{Figs}(a), we compare the residuals of CPP, RCPP and QDGT against the number of iterations. It can be seen that the performance of CPP deteriorates due to the absolute compression error. Meanwhile, RCPP outperforms QDGT under different compression methods. 

   From Fig. \ref{Figs}(b) where we further compare the performance of the algorithms against the communication bits, we find that RCPP performs better than the other methods under different compression methods. Besides, RCPP with QTn achieves the best communication efficiency. This implies that by considering Assumption \ref{Assumption:General} which provides us with more choices on the compression operators, RCPP may achieve better performance under a specific choice of compression method with less communication (which may not satisfy the previous assumptions).
\begin{figure}[htbp]
	\centering
	\vspace{-1.5em}
	\subfigure[]{ \includegraphics[scale=0.31]{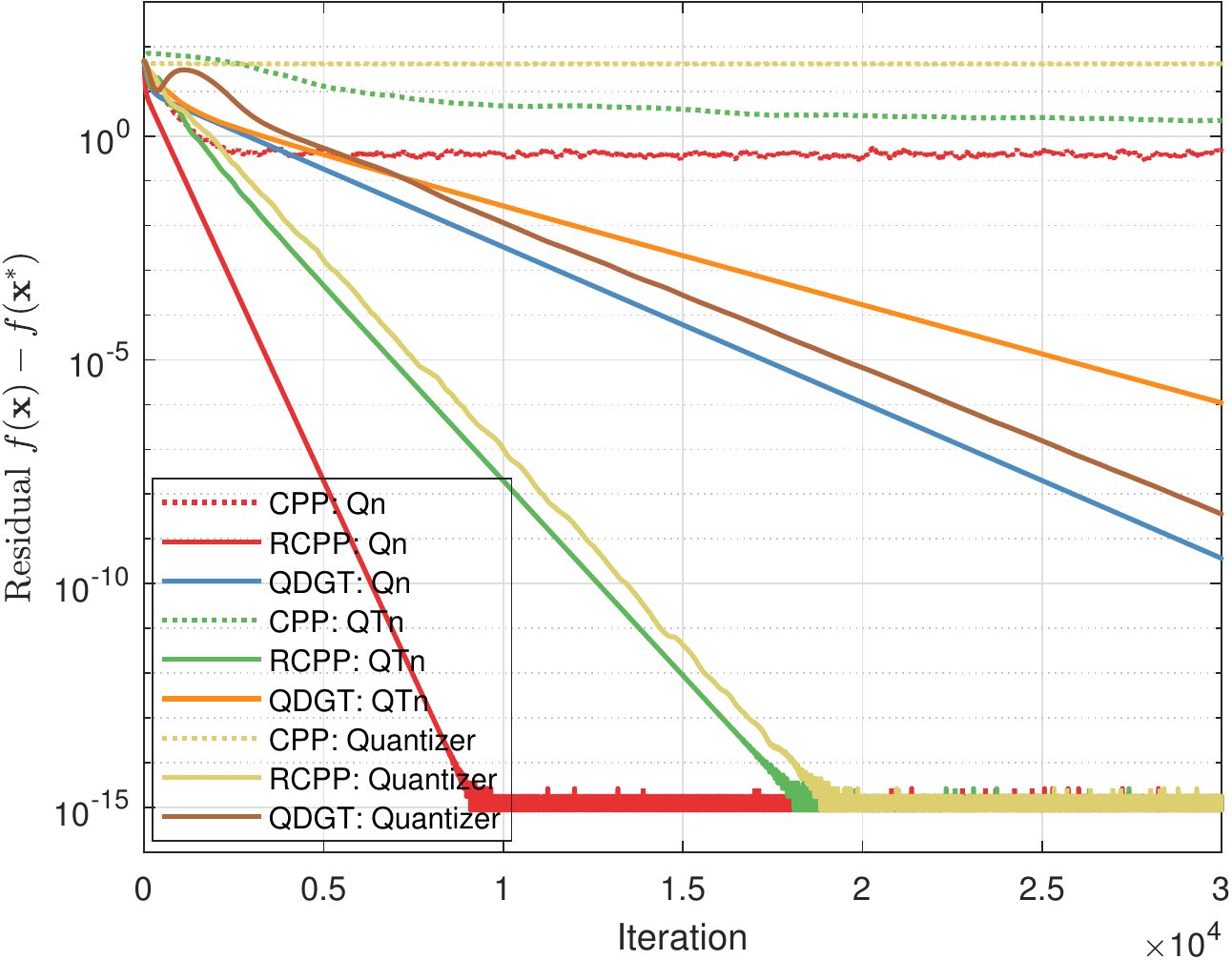} }%
    \hfill
	\subfigure[]{ \includegraphics[scale=0.31]{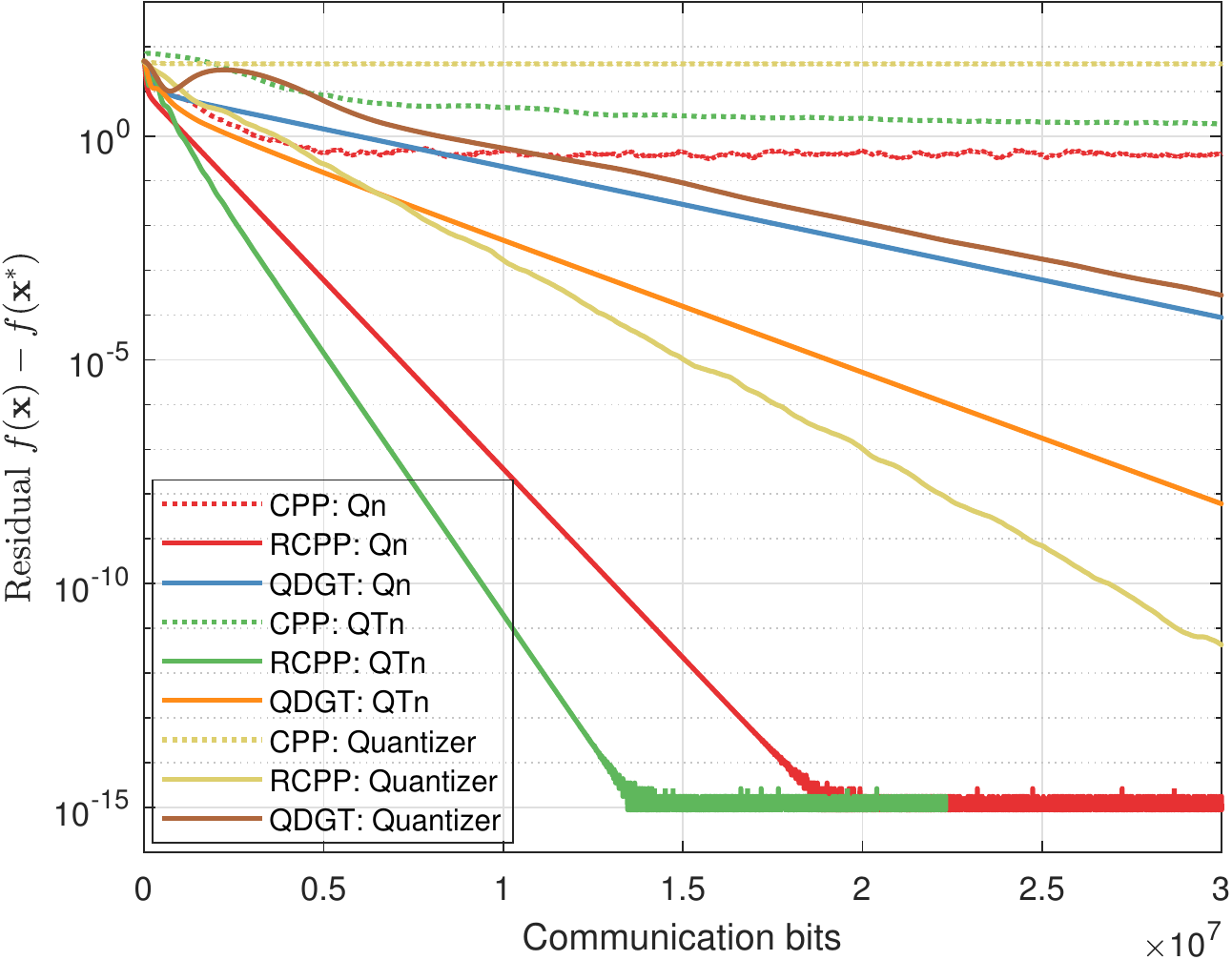} }
    \vspace{-1em}
	\caption{Residuals $\EE\big[f(\ox^{k})-f(\vx^*)\big]$ against the number of iterations and communication bits respectively for CPP, RCPP and QDGT under different compression methods.}
	\label{Figs}
\end{figure}

\section{Conclusions}\label{Sec: Conclusion}
    This article considers decentralized optimization with communication compression over directed networks. Specifically, we consider a general class of compression operators that allow both relative and  absolute compression errors. For smooth objective functions satisfying the PL condition, we propose a robust compressed push-pull algorithm, termed RCPP, which converges linearly. Numerical results demonstrate that RCPP is efficient and robust to various compressors.
    % Future work will consider analyzing the effect of  privacy protection for RCPP and explore how to achieve intermittent gradient tracking to further save communication.
    
%%%%%%%%%%%%%%%%%%%%%%%%%%%%%%%%%%%%%%%%%%%%%%%%%%%%%%%%%%%%%%%%%%%%%%%%%%%%%%%%%%%%%%%%%%%%%%%%%%	
    \bibliographystyle{IEEEtran}
    \bibliography{lywbib,RCPP}	
%%%%%%%%%%%%%%%%%%%%%%%%%%%%%%%%%%%%%%%%%%%%%%%%%%%%%%%%%%%%%%%%%%%%%%%%%%%%%%%%%%%%%%%%%%%%%%%%%%

\appendix
\section{Proofs}\label{Appendix1}
\subsection{Supplementary Lemmas}
\begin{lemma}\label{lem:UV}
	For  $\vU,\vV\in \RR^{n\times p}$ and any constant $\tau>0$,   we have the following inequality:
	\begin{align}
	\|\vU+\vV\|^2\leq (1+\tau)\|\vU\|^2 + (1+\frac{1}{\tau})\|\vV\|^2.
	\end{align}
	In particular, taking $1+\tau=\frac{1}{1-\tau_1}, 0<\tau_1<1$ and $1+\tau=\tau_2, \tau_2>1$, we have
	\begin{align}\label{ineq:UV1}
	\|\vU+\vV\|^2\leq \frac{1}{1-\tau_1} \|\vU\|^2 + \frac{1}{\tau_1}\|\vV\|^2
	\end{align}
	and 	
	\begin{align}\label{ineq:UV2}
	\|\vU+\vV\|^2\leq \tau_2 \|\vU\|^2 + \frac{\tau_2}{\tau_2-1}\|\vV\|^2.
	\end{align}
	In addition, for any $\vU_1,\vU_2,\vU_3 \in \RR^{n\times p}$ and $\tau'>1$, we have 
	$
	\|\vU_1+\vU_2+\vU_3\|^2\leq \tau'\|\vU_1\|^2 + \frac{2\tau'}{\tau'-1}\left[\|\vU_2\|^2+\|\vU_3\|^2\right]
	$ 	
	and 
	$
	\|\vU_1+\vU_2+\vU_3\|^2\leq 3\|\vU_1\|^2 + 3\|\vU_2\|^2+3\|\vU_3\|^2.
	$
\end{lemma}
\begin{lemma}\label{lem:tau}
	For any $\tau\in\RR_{+}$, there holds
	\begin{align}
	(1+\frac{\tau}{2})(1-\tau)\leq 1-\frac{\tau}{2}
	\end{align}
\end{lemma}
Note that Lemma \ref{lem:UV} and \ref{lem:tau} will be frequently used in the proof for the linear system of inequalities in Lemma \ref{Lem:RCPP}.

\subsection{Some Algebraic Results}
From some simple algebraic operations, we know $\Pi_{R} \vR=\vR\Pi_{R}=\vR-\frac{\vone \vu_{R}^{\T}}{n}$, $\Pi_{C} \vC=\vC\Pi_{C}=\vC-\frac{\vu_{C}\vone ^{\T}}{n}$, $\Pi_{R}\vX^{k}=\vX^{k}-\vone \oX^{k}$, $\Pi_{C}\vY^{k}=\vY^{k}- \vu_{C}\oY^{k}$, $(\vI-\vR)\Pi_{R}=\Pi_{R}(\vI-\vR)=\vI-\vR$, and $\Pi_{C}(\vI-\vC)=(\vI-\vC)\Pi_{C}=\vI-\vC$.
Before deriving the linear system of inequalities in Lemma \ref{Lem:RCPP}, we need some preliminary results on $\oX^{k+1}$, $\Pi_{R}\vX^{k+1}$, and $\Pi_{C}\vY^{k+1}$.

First, from the equivalent formula \eqref{Xkplus}, we know
\begin{align}\label{oXk}	
\nonumber	&\oX^{k+1}\\
\nonumber=&\frac{1}{n}\vu_{R}^{\T}\tX^{k}-\gamma_x \frac{1}{n}\vu_{R}^{\T} (\vI-\vR)\widehat{\vX}^{k}\\
\nonumber	=&\oX^{k}-\frac{1}{n}\vu_{R}^{\T}\Lambda \vY^{k}\\
\nonumber	=&\oX^{k}-\frac{1}{n}\vu_{R}^{\T}\Lambda (\vY^{k}-\vu_{C}\oY^{k}+\vu_{C}\oY^{k})\\
\nonumber	=&\oX^{k}-\overline{\lambda} \oY^{k}-\frac{1}{n}\vu_{R}^{\T}\Lambda \Pi_{C}\vY^{k} \\
\nonumber=&\oX^{k}-\overline{\lambda}\nabla f(\oX^{k})+\overline{\lambda}(\nabla f(\oX^{k})-\oY^{k}) \\
	&-\frac{1}{n}\vu_{R}^{\T}\Lambda \Pi_{C}\vY^{k},
\end{align}
where $\oX^{k}=\frac{1}{n}\vu_{R}^{\T}\vX^{k}$, $\overline{\lambda}=\frac{1}{n}\vu_{R}^\T\Lambda \vu_{C}$ and the fact $\vu_{R}^{\T} (\vI-\vR)=\vzero$ is used in the second equality.

Second, for $\Pi_{R}\vX^{k}=\vX^{k}-\vone \oX^{k}$, we have
\begin{align}
\nonumber	&\Pi_{R}\vX^{k+1}\\
\nonumber=&\Pi_{R}\tX^{k}-\gamma_x \Pi_{R} (\vI-\vR)\widehat{\vX}^{k}\\
\nonumber	=&(\vI-\gamma_x(\vI-\vR))\Pi_{R}\tX^{k}+\gamma_x  (\vI-\vR)\Pi_{R}(\tX^{k}-\widehat{\vX}^{k})\\
\nonumber	=&\Pi_{R}\vR_{\gamma}\Pi_{R}\vX^{k}-\Pi_{R}\vR_{\gamma}\Lambda\vY^{k}\\
\label{PiRXk}&+\gamma_x  (\vI-\vR)(\tX^{k}-\widehat{\vX}^{k}),
\end{align}
where $\vR_{\gamma}=\vI-\gamma_x(\vI-\vR)$ and the relation
% $\Pi_{R}\vR_{\gamma}=\Pi_{R}-\gamma_x\Pi_{R}(\vI-\vR)=\vR_{\gamma}=\vR_{\gamma}\Pi_{R}$ is adopted in the third equality.
$\Pi_{R}\vR_{\gamma}\Pi_{R}=\Pi_{R}\vR_{\gamma}=\vR_{\gamma}\Pi_{R}$ is adopted for the third equality.

Finally, for $\Pi_{C}\vY^{k}=\vY^{k}-\vu_{C} \oY^{k}$, we have
\begin{align}
\nonumber \Pi_{C}\vY^{k+1}=&\Pi_{C}(\tY^{k}-\gamma_{y}(\vI-\vC)\widehat{\vY}^{k})\\
\nonumber =&(\vI-\gamma_{y}(\vI-\vC))\Pi_{C}\tY^{k}+\gamma_{y}(\vI-\vC)(\tY^{k}-\widehat{\vY}^{k})\\
\nonumber=&\Pi_{C}\vC_{\gamma}\Pi_{C}\vY^{k}+\Pi_{C}\vC_{\gamma}(\nabla\vF(\vX^{k+1})-\nabla\vF(\vX^{k}))\\
\label{PiCYk}&+\gamma_{y}(\vI-\vC)(\tY^{k}-\widehat{\vY}^{k}),
\end{align}
where $\vC_{\gamma}=\vI-\gamma_y(\vI-\vC)$ and the relation
% $\Pi_{C}\vC_{\gamma}\Pi_{C}=\Pi_{C}(\Pi_{C}-\gamma_y(\vI-\vC)\Pi_{C})=\Pi_{C}\Pi_{C}\vC_{\gamma}=\Pi_{C}\vC_{\gamma}$
$\Pi_{C}\vC_{\gamma}\Pi_{C}=\Pi_{C}\vC_{\gamma}=\vC_{\gamma}\Pi_{C}$ is adopted for the third equality.

\subsection{Proof of Lemma \ref{Lem:Yk}}\label{Pf:Yk}
Based on Lemma \ref{lem:UV}, we obtain
\begin{align}\label{YkNorm}
\nonumber\norm{\vY^{k}}_{F}^{2}
=&\norm{\vY^{k}-\vu_{C}\oY^{k}+\vu_{C}(\oY^{k}-\nabla f(\oX^{k}))+\vu_{C}\nabla f(\oX^{k})}_{F}^{2}\\
\nonumber	\leq& 3\norm{ \vY^{k}-\vu_{C}\oY^{k} }_{F}^{2}+3\norm{ \vu_{C}(\oY^{k}-\nabla f(\oX^{k}))}_{F}^{2}\\
\nonumber&+3\norm{\vu_{C}\nabla f(\oX^{k}) }_{F}^{2}\\
\nonumber\leq& 3\norm{ \Pi_{C}\vY^{k} }_{C}^{2}+\frac{3\norm{ \vu_{C} }^{2}}{n} L^2 \norm{ \Pi_{R}\vX^{k} }_{R}^{2} \\
&+3 \norm{\vu_{C}}^{2}\norm{\nabla f(\oX^{k})}^{2},
\end{align}
where we use the fact that
\begin{align}
\nonumber	&\norm{\nabla f(\oX^{k})-\oY^{k}}\\
\nonumber	\leq& \frac{1}{n}\norm{\vone^{\T}\nabla\vF(\vone \oX^{k})-\vone^{\T}\nabla\vF( \vX^{k})}\\
\leq&\frac{L}{\sqrt{n}}\norm{\vX^{k}-\vone \oX^{k}}_{F}\leq \frac{L}{\sqrt{n}}\norm{\Pi_{R}\vX^{k}}_{R}.
\end{align}

\subsection{Proof of Lemma \ref{Lem:RCPP}}\label{Pf:LMI}
For simplicity, denote $\cF^k$ as the $\sigma$-algebra generated by $\{\vX^0,\vY^0,\vX^1,\vY^1,\cdots,\vX^{k},\vY^{k}\}$, and define $\EE[ \cdot |\cF^k]$ as the conditional expectation with respect to the compression operator given $\cF^k$.
\subsubsection{First inequality}
From \eqref{PiRXk}, we have
\begin{align}\label{PiRXkNorm}
\nonumber	&\norm{\Pi_{R}\vX^{k+1}}_{R}^{2}\\
\nonumber	=&\norm{\Pi_{R}\vR_{\gamma}\Pi_{R}\vX^{k}-\Pi_{R}\vR_{\gamma}\Lambda\vY^{k}+\gamma_x  (\vI-\vR)(\tX^{k}-\widehat{\vX}^{k})}_{R}^{2}\\
\nonumber	\leq&(1-\theta_{R}\gamma_x)\norm{\Pi_{R}\vX^{k}}_{R}^{2}\\
\nonumber &+\frac{2}{\theta_{R}\gamma_x}\Big( (1-\theta_{R}\gamma_x)^{2}\norm{\Lambda\vY^{k}}_{R}^{2}\\ 
\nonumber &+\gamma_x^2\norm{\vI-\vR}_{R}^{2}\norm{\tX^{k}-\widehat{\vX}^{k}}_{R}^{2}\Big)\\
\nonumber\leq&(1-\theta_{R}\gamma_x)\norm{\Pi_{R}\vX^{k}}_{R}^{2}\\
\nonumber&+\frac{2}{\theta_{R}\gamma_x}\Big( (1-\theta_{R}\gamma_x)^{2}\delta_{R,2}^2\norm{\Lambda}_{2}^{2}\norm{\vY^{k}}_{F}^{2}\\  &+\delta_{R,2}^2\gamma_x^2\norm{\vI-\vR}_{R}^{2}\norm{\tX^{k}-\widehat{\vX}^{k}}_{F}^{2}\Big),
\end{align}
where the first inequality is due to Lemma \ref{lem:UV} with $\tau_1=\theta_{R}\gamma_x$ and $\norm{\Pi_{R}\vR_{\gamma}}_{R}\leq 1-\theta_{R}\gamma_x$. 

Taking the conditional expectation on $\cF^{k}$ yields
\begin{align}\label{PiXkb1}
\nonumber	&\EE[\norm{\Pi_{R}\vX^{k+1}}_{R}^{2}\vert \cF^{k}]\\
\nonumber	\leq&(1-\theta_{R}\gamma_x)\norm{\Pi_{R}\vX^{k}}_{R}^{2}\\
\nonumber &+\frac{2}{\theta_{R}\gamma_x}\Bigg[ (1-\theta_{R}\gamma_x)^{2}\delta_{R,2}^2\norm{\Lambda}_{2}^{2}\norm{\vY^{k}}_{F}^{2}\\ \nonumber&+\delta_{R,2}^2\gamma_x^2\norm{\vI-\vR}_{R}^{2}C\norm{\tX^{k}-\vH_{x}^{k}}_{F}^{2}\\
&+\delta_{R,2}^2\gamma_x^2\norm{\vI-\vR}_{R}^{2} s_k^{2}\sigma^2\Bigg],
\end{align}
where we use the fact that
\begin{align}\label{tXk_hXkb}
\EE[\norm{\tX^{k}-\widehat{\vX}^{k}}_{F}^{2} \vert \cF^{k}]\leq& C \norm{\tX^{k}-\vH_{x}^{k}}_{F}^{2}+s_k^{2}\sigma^2.
\end{align}
For convenience, we rewrite \eqref{PiXkb1} as
\begin{align}\label{PiXkb2}
\nonumber	&\EE[\norm{\Pi_{R}\vX^{k+1}}_{R}^{2}\vert \cF^{k}]\\
\nonumber \leq&(1-\theta_{R}\gamma_x)\norm{\Pi_{R}\vX^{k}}_{R}^{2}\\
\nonumber	&+\frac{1}{\theta_{R}\gamma_x}\Bigg[ c_2(1-\theta_{R}\gamma_x)^{2}\hat{\lambda}^{2}\norm{\vY^{k}}_{F}^{2}\\ &+c_{24}C\gamma_x^2\norm{\tX^{k}-\vH_{x}^{k}}_{F}^{2}\Bigg]+ s_k^{2}\zeta_{c},
\end{align}
where $c_2= \delta_{R,2}^2$,
$c_{24}=2\delta_{R,2}^2\norm{\vI-\vR}_{R}^{2}$ and 
$\zeta_{c}= 2\sigma^2/\theta_{R}\delta_{R,2}^2\norm{\vI-\vR}_{R}^{2}
\ge 2\sigma^2\delta_{R,2}^2\norm{\vI-\vR}_{R}^{2}\gamma_x^2/(\theta_{R}\gamma_x)$.

\subsubsection{Second  inequality}
By relation \eqref{PiCYk}, we obtain
\begin{align}\label{PiCYkNorm}
\nonumber	&\norm{\Pi_{C}\vY^{k+1}}_{C}^{2}\\
\nonumber	=&\Big\Vert \Pi_{C}\vC_{\gamma}\Pi_{C}\vY^{k}+\Pi_{C}\vC_{\gamma}(\nabla\vF(\vX^{k+1})-\nabla\vF(\vX^{k}))\\
\nonumber &	+\gamma_{y}(\vI-\vC)(\tY^{k}-\widehat{\vY}^{k})\Big\Vert_{C}^{2}\\
\nonumber	\leq&(1-\theta_{C}\gamma_y)\norm{\Pi_{C}\vY^{k}}_{C}^{2}\\
\nonumber &+\frac{2}{\theta_{C}\gamma_y}\Big( (1-\theta_{C}\gamma_y)^{2}\norm{\nabla\vF(\vX^{k+1})-\nabla\vF(\vX^{k})}_{C}^{2}\\ 
\nonumber & +\gamma_y^2\norm{(\vI-\vC)\tY^{k}-\widehat{\vY}^{k}}_{C}^{2}\Big)\\
\nonumber	\leq&(1-\theta_{C}\gamma_y)\norm{\Pi_{C}\vY^{k}}_{C}^{2}\\
\nonumber &+\frac{2}{\theta_{C}\gamma_y}\Big( (1-\theta_{C}\gamma_y)^{2}\delta_{C,2}^2\norm{\nabla\vF(\vX^{k+1})-\nabla\vF(\vX^{k})}_{F}^{2}\\ &+\delta_{C,2}^2\gamma_y^2\norm{\vI-\vC}_{C}^{2}\norm{\tY^{k}-\widehat{\vY}^{k}}_{F}^{2}\Big). 
\end{align}
Note that 
\begin{align}\label{tYk_hYk}
\nonumber	&\EE[\norm{\tY^{k}-\widehat{\vY}^{k}}_{F}^{2} \vert \cF^{k}]\\
\nonumber\leq& C \EE[\norm{\tY^{k}-\vH_{y}^{k}}_{F}^{2}\vert \cF^{k}]+s_k^{2}\sigma^2\\
\nonumber \leq& 2  C\norm{\vY^{k}-\vH_{y}^{k}}_{F}^{2}\\
& +2 C\EE[\norm{\nabla\vF(\vX^{k+1})-\nabla\vF(\vX^{k})}_{F}^{2}\vert \cF^{k}]+s_k^{2}\sigma^2.
\end{align}
Taking the conditional expectation on both sides of \eqref{PiCYkNorm}, we get
\begin{align}
\nonumber	&\EE[\norm{\Pi_{C}\vY^{k+1}}_{C}^{2}\vert \cF^{k}]\\
\nonumber	\leq&(1-\theta_{C}\gamma_y)\norm{\Pi_{C}\vY^{k}}_{C}^{2}\\
\nonumber &+\frac{2}{\theta_{C}\gamma_y}\Bigg[ 2C\delta_{C,2}^2\gamma_y^2\norm{\vI-\vC}_{2}^{2}\norm{\vY^{k}-\vH_{x}^{k}}_{F}^{2}
+\Big((1-\theta_{C}\gamma_y)^{2}\\ 
\nonumber &+2C\gamma_y^2\norm{\vI-\vC}_{C}^{2}\Big)\delta_{C,2}^2
\EE[\norm{\nabla\vF(\vX^{k+1})-\nabla\vF(\vX^{k})}_{F}^{2}\vert \cF^{k}]\\ 
\label{PiCYkNorm1} &
+\delta_{C,2}^2\gamma_y^2\norm{\vI-\vC}_{C}^{2}s_k^{2}\sigma^2\Bigg].
\end{align}
Recalling that 
\begin{align}\label{nFp_nFb}
\norm{\nabla\vF(\vX^{k+1})-\nabla\vF(\vX^{k})}_{F}^{2}\leq L^2 \norm{\vX^{k+1}-\vX^{k}}_{F}^{2},
\end{align}
we need to bound $\norm{\vX^{k+1}-\vX^{k}}_{F}^{2}$. From the update of the decision variables in Algorithm \ref{Alg:RCGTe}, we know
\begin{align}
\nonumber	\vX^{k+1}-\vX^{k}=&\widetilde{\vX}^{k}-\gamma_{x}(\vI-\vR)\widehat{\vX}^{k}-\vX^{k}\\
\nonumber=&\gamma_{x}(\vI-\vR)(\widetilde{\vX}^{k}-\widehat{\vX}^{k})-\gamma_{x}(\vI-\vR)\Pi_{R}\vX^{k}\\
&-\vR_{\gamma}\Lambda \vY^{k}.
\end{align}
Thus, we have 
\begin{align}\label{Xkp_XkNorm}	
\nonumber	&\norm{\vX^{k+1}-\vX^{k}}_{F}^{2}\\
\nonumber	\leq&3\norm{\gamma_{x}(\vI-\vR)(\widetilde{\vX}^{k}-\widehat{\vX}^{k})}_{F}^{2}
+3\norm{\gamma_{x}(\vI-\vR)\Pi_{R}\vX^{k}}_{F}^{2}\\
\nonumber	&+3\norm{\vR_{\gamma}\Lambda \vY^{k}}_{F}^{2}\\
\nonumber	\leq&3\gamma_{x}^{2}\norm{\vI-\vR}_{R}^{2}\norm{\widetilde{\vX}^{k}-\widehat{\vX}^{k}}_{F}^{2}
+3\gamma_{x}^{2}\norm{\vI-\vR}_{R}^{2}\norm{\Pi_{R}\vX^{k}}_{R}^{2}\\
&+3\hlambda^2\norm{\vR_{\gamma}}_{R}^{2}\norm{ \vY^{k}}_{F}^{2}.
\end{align}
Putting \eqref{Xkp_XkNorm}  and  \eqref{nFp_nFb}  back into \eqref{PiCYkNorm1} and using relation \eqref{tXk_hXkb}, we have
\begin{align}\label{PiCYkNorm2}
\nonumber	&\EE[\norm{\Pi_{C}\vY^{k+1}}_{C}^{2}\vert \cF^{k}]\\
\nonumber	\leq&(1-\theta_{C}\gamma_y)\norm{\Pi_{C}\vY^{k}}_{C}^{2}
+\frac{2}{\theta_{C}\gamma_y}\Bigg[ \Big((1-\theta_{C}\gamma_y)^{2}\\
\nonumber   &+2C\gamma_y^2\norm{\vI-\vC}_{C}^{2}\Big)\delta_{C,2}^2
L^2\Bigg(
3\gamma_{x}^{2}\norm{\vI-\vR}_{R}^{2}C\norm{\widetilde{\vX}^{k}-\vH_{x}^{k}}_{F}^{2}\\
\nonumber	&+3\gamma_{x}^{2}\norm{\vI-\vR}_{R}^{2}s_k^{2}\sigma^2
+3\gamma_{x}^{2}\norm{\vI-\vR}_{R}^{2}\norm{\Pi_{R}\vX^{k}}_{F}^{2}\\
\nonumber&+3\hlambda^2\norm{\vR_{\gamma}}_{R}^{2}\norm{ \vY^{k}}_{F}^{2} \Bigg)\\
\nonumber&+2C\delta_{C,2}^2\gamma_y^2\norm{\vI-\vC}_{C}^{2}\norm{\vY^{k}-\vH_{x}^{k}}_{F}^{2}\\
&+\delta_{C,2}^2\gamma_y^2\norm{\vI-\vC}_{C}^{2}s_k^{2}\sigma^2\Bigg].
\end{align}
For simplicity, we reorganize \eqref{PiCYkNorm2} as
\begin{align}\label{PiCYkNorm3}
\nonumber	&\EE[\norm{\Pi_{C}\vY^{k+1}}_{C}^{2}\vert \cF^{k}]\\
\nonumber	\leq&(1-\theta_{C}\gamma_y)\norm{\Pi_{C}\vY^{k}}_{C}^{2}\\
\nonumber	&+\frac{1}{\theta_{C}\gamma_y}\Bigg[ e_1
L^2\Bigg(
c_{34}C\gamma_{x}^{2}\norm{\widetilde{\vX}^{k}-\vH_{x}^{k}}_{F}^{2}
+c_{32}\gamma_{x}^{2}\norm{\Pi_{R}\vX^{k}}_{F}^{2}\\
&+ c_3\hlambda^2\norm{ \vY^{k}}_{F}^{2}  \Bigg)
+c_{35}C \gamma_y^2 \norm{\vY^{k}-\vH_{x}^{k}}_{F}^{2}\Bigg]
+s_k^{2}\zeta_{g},
\end{align}
where $c_3=3\norm{\vR_{\gamma}}_{R}^{2}$, 
$d_1=2\delta_{C,2}^2$, 
$d_2=4\delta_{C,2}^2\norm{\vI-\vC}_{C}^{2}$, 
$e_1=d_1 +d_2C  \geq d_1(1-\theta_{C}\gamma_y)^{2}+d_2C\gamma_y^2$, 
$c_{32}=c_{34}=3\norm{\vI-\vR}_{R}^{2}$, 
$c_{35}=4 \delta_{C,2}^2 \norm{\vI-\vC}_{C}^{2}$, 
and 
$\zeta_{g}=6\sigma^2/\theta_{C} \norm{\vI-\vR}_{R}^{2} L^2\big(d_1+d_2C\big) +2\sigma^2/\theta_{C}\delta_{C,2}^2\norm{\vI-\vC}_{C}^{2} 
\ge 6\sigma^2 \norm{\vI-\vR}_{R}^{2} L^2\big(d_1(1-\theta_{C}\gamma_y)^{2}+d_2C\gamma_y^2\big)\gamma_{x}^{2}/(\theta_{C}\gamma_y) +2\sigma^2\delta_{C,2}^2\norm{\vR_{\gamma}}_{R}^{2}\norm{\vI-\vC}_{C}^{2} \gamma_y^2/(\theta_{C}\gamma_y)$ 
holds since $\gamma_x\leq\gamma_y$.

\subsubsection{Third  inequality}
Recalling the  update of the variables $\vH_{x}^{k+1}$, $\tX^{k}$, and $\vY^{k+1}$ in Algorithm \ref{Alg:RCGTe}, we have
\begin{equation*}
\begin{aligned}
&\tX^{k+1}-\vH_{x}^{k+1}\\
=&\tX^{k+1}-\tX^{k}+\tX^{k}-(\vH^{k}_{x}+\alpha_x\vQ^{k}_{x})\\
=&\tX^{k+1}-\tX^{k}+(1-\alpha_x r)(\tX^{k}-\vH^{k}_{x})\\
&+\alpha_x r(\tX^{k}-\vH^{k}_{x}- \frac{\vQ^{k}_{x}}{r})	\\
=&\vX^{k+1}-\vX^{k}-\Lambda(\vY^{k+1}-\vY^{k})\\
&+(1-\alpha_x r)(\tX^{k}-\vH^{k}_{x})+\alpha_x r(\tX^{k}-\vH^{k}_{x}- \frac{\vQ^{k}_{x}}{r})
\end{aligned}
\end{equation*}
and 
\begin{align}\label{Ykp_Yk}
\nonumber	&\vY^{k+1}-\vY^{k}\\
\nonumber=&\tY^{k}-\gamma_{y} (\vI-\vC)\widehat{\vY}^{k}-\vY^{k}\\
\nonumber=&\gamma_{y}(\vI-\vC)(\tY^{k}-\widehat{\vY}^{k})-\gamma_{y}(\vI-\vC)\Pi_{C}\vY^{k}\\
&+\vC_{\gamma}(\nabla\vF(\vX^{k+1})-\nabla\vF(\vX^{k})).
\end{align}
Based on Lemma \ref{lem:UV},  we get
\begin{align}
\nonumber	&\norm{\tX^{k+1}-\vH_{x}^{k+1}}_{F}^{2}\\
\nonumber\leq&(1+\frac{2}{\alpha_x r \delta})\norm{\vX^{k+1}-\vX^{k}-\Lambda(\vY^{k+1}-\vY^{k})}_{F}^{2}\\
\nonumber&+(1+\frac{\alpha_x r \delta}{2})\bigg\Vert(1-\alpha_x r)(\tX^{k}-\vH^{k}_{x})\\
\nonumber&~+\alpha_x r(\tX^{k}-\vH^{k}_{x}- \frac{\vQ^{k}_{x}}{r})\bigg\Vert_{F}^{2}.
\end{align}
Taking the conditional expectation on $\cF^{k}$ yields
\begin{align}\label{tXkp_Hxkp}
\nonumber&\EE[\norm{\tX^{k+1}-\vH_{x}^{k+1}}_{F}^{2} \vert \cF^{k}]\\
\nonumber	\leq & (1+\frac{2}{\alpha_x r \delta})\Big(2\EE[\norm{\vX^{k+1}-\vX^{k}}_{F}^{2}\vert \cF^{k}]\\
\nonumber&+2\EE[\norm{\Lambda}_{2}^{2}\norm{\vY^{k+1}-\vY^{k})}_{F}^{2}\vert \cF^{k}]\Big)\\
\nonumber	&+(1+\frac{\alpha_x r \delta}{2})(1-\alpha_x r \delta)\norm{\tX^{k}-\vH^{k}_{x}}_{F}^{2}\\
&+s_k^{2}\alpha_x r \sigma^2_{r}(1+\frac{\alpha_x r \delta}{2}),
\end{align}
where we use 
\begin{align*}
&\EE[\norm{(1-\alpha_x r)(\tX^{k}-\vH^{k}_{x})+\alpha_x r(\tX^{k}-\vH^{k}_{x}- \frac{\vQ^{k}_{x}}{r})}_{F}^{2}\vert \cF^{k}]\\
\leq&(1-\alpha_x r)\norm{ \tX^{k}-\vH^{k}_{x} }_{F}^{2}+\alpha_x r \EE[\norm{ \tX^{k}-\vH^{k}_{x}- \frac{\vQ^{k}_{x}}{r} }_{F}^{2}\vert \cF^{k}]\\
\leq& (1-\alpha_x r)\norm{ \tX^{k}-\vH^{k}_{x} }_{F}^{2}+\alpha_x r(1-\delta)\norm{ \tX^{k}-\vH^{k}_{x} }_{F}^{2} \\
&+\alpha_x r s_k^2 \sigma^2_{r}\\
=& (1-\alpha_x r\delta)\norm{ \tX^{k}-\vH^{k}_{x} }_{F}^{2}  +s_k^2\alpha_x r \sigma^2_{r}.
\end{align*}
Then, we bound $\EE[\norm{\vY^{k+1}-\vY^{k}}_{F}^{2}\vert \cF^{k}]$.  Repeatedly using Lemma \ref{lem:UV} together with relation \eqref{Ykp_Yk}, we have 
\begin{align}
\nonumber &\EE[\norm{\vY^{k+1}-\vY^{k}}_{F}^{2}\vert \cF^{k}]\\
\nonumber\leq&3 \EE[\norm{ \gamma_{y}(\vI-\vC)(\tY^{k}-\widehat{\vY}^{k})}_{F}^{2}\vert \cF^{k}]
+3 \norm{\gamma_{y}(\vI-\vC)\Pi_{C}\vY^{k} }_{F}^{2}\\
\nonumber&+3 \EE[\norm{ \vC_{\gamma}(\nabla\vF(\vX^{k+1})-\nabla\vF(\vX^{k}))}_{F}^{2}\vert \cF^{k}]\\
\nonumber\leq&3\gamma_{y}^2\norm{ \vI-\vC}_{C}^{2}\EE[\norm{\tY^{k}-\widehat{\vY}^{k}}_{F}^{2}\vert \cF^{k}]\\
\nonumber&+3\gamma_{y}^2\norm{ \vI-\vC}_{C}^{2} \norm{ \Pi_{C}\vY^{k} }_{F}^{2}\\
\nonumber&+3\norm{ \vC_{\gamma} }_{C}^{2} \EE[\norm{  \nabla\vF(\vX^{k+1})-\nabla\vF(\vX^{k})}_{F}^{2}\vert \cF^{k}].
\end{align}
%Using $\EE[\norm{\tY^{k}-\widehat{\vY}^{k}}_{F}^{2}\vert \cF^{k}]\leq C\EE[\norm{\tY^{k}-\vH_{y}^{k}}_{F}^{2}\vert \cF^{k}]+s_k^2\sigma^2 $, we obtain 
Using \eqref{tYk_hYk}, we obtain
\begin{align}\label{Ykp_Yk1}
\nonumber &\EE[\norm{\vY^{k+1}-\vY^{k}}_{F}^{2}\vert \cF^{k}]\\
\nonumber \leq&3\gamma_{y}^2\norm{ \vI-\vC}_{2}^{2}\Big( 2C\norm{\vY^{k}-\vH_{y}^{k}}_{F}^{2}\\ \nonumber&+2C\EE[\norm{  \nabla\vF(\vX^{k+1})-\nabla\vF(\vX^{k})}_{F}^{2}\vert \cF^{k}]+s_{k}^2\sigma^2\Big)\\
\nonumber&
+3\gamma_{y}^2\norm{ \vI-\vC}_{C}^{2} \norm{ \Pi_{C}\vY^{k} }_{F}^{2}\\
\nonumber&+3\norm{ \vC_{\gamma} }_{C}^{2} \EE[\norm{  \nabla\vF(\vX^{k+1})-\nabla\vF(\vX^{k})}_{F}^{2}\vert \cF^{k}]\\
\nonumber\leq&6C\gamma_{y}^2\norm{ \vI-\vC}_{C}^{2}  \norm{\vY^{k}-\vH_{y}^{k}}_{F}^{2} 
+3\gamma_{y}^2\norm{ \vI-\vC}_{C}^{2}s_{k}^2\sigma^2\\
\nonumber&+3\gamma_{y}^2\norm{ \vI-\vC}_{C}^{2} \norm{ \Pi_{C}\vY^{k} }_{F}^{2}
+\Big(3\norm{ \vC_{\gamma} }_{C}^{2}\\
&~+6C\gamma_{y}^2\norm{ \vI-\vC}_{C}^{2}\Big)L^2 \EE[\norm{\vX^{k+1}-\vX^{k}}_{F}^{2}\vert \cF^{k}].
\end{align}
Substituting \eqref{Ykp_Yk1} into \eqref{tXkp_Hxkp}, we have
\begin{align}
\nonumber&\EE[\norm{\tX^{k+1}-\vH_{x}^{k+1}}_{F}^{2} \vert \cF^{k}]\\
\nonumber \leq& (1+\frac{2}{\alpha_x r \delta})\Bigg[\Big(2+6L^2\hlambda^2 \norm{\vC_{\gamma}}_{C}^{2}\\
\nonumber&~+12CL^2\hlambda^2  \gamma_{y}^2 \norm{\vI-\vC}_{C}^{2}\Big)\EE[\norm{\vX^{k+1}-\vX^{k}}_{F}^{2}\vert \cF^{k}]\\
\nonumber	&+12C\hlambda^2  \gamma_{y}^2 \norm{\vI-\vC}_{2}^{2} \norm{ \vY^{k}-\vH_{y}^{k}}_{F}^{2}\\
\nonumber	&+6\hlambda^2 \gamma_{y}^2 \norm{\vI-\vC}_{2}^{2}\norm{\Pi_{C}\vY^{k} }_{C}^{2}
\Bigg]\\
\nonumber	&+(1-\frac{\alpha_x r \delta}{2})\norm{\tX^{k}-\vH^{k}_{x}}_{F}^{2}+s_k^{2}\alpha_x r \sigma^2_{r}(1+\frac{\alpha_x r \delta}{2})\\
&+6s_{k}^2\sigma^2\hlambda^2  \gamma_{y}^2 \norm{\vI-\vC}_{C}^{2} (1+\frac{2}{\alpha_x r \delta}).
\end{align}

Noting that  $6L^2\hlambda^2 \norm{\vC_{\gamma}}_{C}^{2}\leq 1 $, $12CL^2\hlambda^2   \norm{\vI-\vC}_{C}^{2}\leq 1 $
and $\gamma_{y}^2\leq 1$ from the assumption, we have  $2+6L^2\hlambda^2 \norm{\vC_{\gamma}}_{C}^{2} +12CL^2\hlambda^2  \norm{\vI-\vC}_{C}^{2}\gamma_{y}^2\leq 4$. 
Based on $1+\frac{2}{\alpha_x r \delta}\leq \frac{3}{\alpha_x r \delta}$ and $1+\frac{\alpha_x r \delta}{2}\leq 2$, we obtain
\begin{align}\label{tXkp_Hxkp2}
\nonumber&\EE[\norm{\tX^{k+1}-\vH_{x}^{k+1}}_{F}^{2} \vert \cF^{k}]\\
\nonumber\leq& \frac{3}{\alpha_x r \delta}\Big(4\EE[\norm{\vX^{k+1}-\vX^{k}}_{F}^{2}\vert \cF^{k}]\\
\nonumber&+12C\hlambda^2 \gamma_{y}^2   \norm{\vI-\vC}_{C}^{2} \norm{ \vY^{k}-\vH_{y}^{k}}_{F}^{2}\\
\nonumber&+6\hlambda^2  \gamma_{y}^2 \norm{\vI-\vC}_{C}^{2}\norm{\Pi_{C}\vY^{k} }_{C}^{2}
\Big)\\
\nonumber&+(1-\frac{\alpha_x r \delta}{2})\norm{\tX^{k}-\vH^{k}_{x}}_{F}^{2}\\
&+2s_k^{2}\alpha_x r \sigma^2_{r}
+s_{k}^2\sigma^2\hlambda^2  \gamma_{y}^2 \norm{\vI-\vC}_{C}^{2} \frac{18}{\alpha_x r \delta}.
\end{align}
Plugging \eqref{Xkp_XkNorm} into \eqref{tXkp_Hxkp2}, we get
\begin{align}
\nonumber&\EE[\norm{\tX^{k+1}-\vH_{x}^{k+1}}_{F}^{2} \vert \cF^{k}]\\
\nonumber\leq& \frac{3}{\alpha_x r \delta}\Bigg(12 \gamma_{x}^{2}\norm{\vI-\vR}_{R}^{2}C\norm{\widetilde{\vX}^{k}-\vH_{x}^{k}}_{F}^{2} \\
\nonumber&+12\gamma_{x}^{2}\norm{\vI-\vR}_{R}^{2}s_k^{2}\sigma^2
+12\gamma_{x}^{2}\norm{\vI-\vR}_{R}^{2}\norm{\Pi_{R}\vX^{k}}_{R}^{2}\\
\nonumber&+12\hlambda^2\norm{\vR_{\gamma}}_{R}^{2}\norm{ \vY^{k}}_{F}^{2}
+12C\hlambda^2 \gamma_{y}^2   \norm{\vI-\vC}_{C}^{2} \norm{ \vY^{k}-\vH_{y}^{k}}_{F}^{2}\\
\nonumber&+6\hlambda^2 \gamma_{y}^2  \norm{\vI-\vC}_{C}^{2}\norm{\Pi_{C}\vY^{k} }_{C}^{2}
\Bigg)\\
\nonumber&+(1-\frac{\alpha_x r \delta}{2})\norm{\tX^{k}-\vH^{k}_{x}}_{F}^{2}\\
\nonumber&+2s_k^{2}\alpha_x r \sigma^2_{r}
+s_{k}^2\sigma^2 \hlambda^2  \gamma_{y}^2 \norm{\vI-\vC}_{C}^{2}   \frac{18}{\alpha_x r \delta}\\
\nonumber=& \frac{1}{\alpha_x r \delta}\Bigg(d_{44} C\gamma_{x}^{2}\norm{\widetilde{\vX}^{k}-\vH_{x}^{k}}_{F}^{2} 
+d_{42}\gamma_{x}^{2}\norm{\Pi_{R}\vX^{k}}_{R}^{2}\\
\nonumber&+ c_4\hlambda^2\norm{ \vY^{k}}_{F}^{2}
+ c_{43}\hlambda^2\gamma_{y}^2\norm{\Pi_{C}\vY^{k}}_{C}^{2}\\
\nonumber&+c_{45}C\hlambda^2  \gamma_{y}^2  \norm{ \vY^{k}-\vH_{y}^{k}}_{F}^{2}
\Bigg)\\
\label{tXkp_Hxkp4}&+(1-\frac{\alpha_x r \delta}{2})\norm{\tX^{k}-\vH^{k}_{x}}_{F}^{2}+s_k^{2}\zeta_{cx}
,
\end{align}
where 
$c_4=12\norm{\vR_{\gamma}}_{R}^{2}$, 
$d_{42}=d_{44}=36\norm{\vI-\vR}_{R}^{2}$, 
$c_{43}= 18\norm{\vI-\vC}_{C}^{2}$, 
$c_{45}=36\norm{\vI-\vC}_{C}^{2}$, 
$\zeta_{cx}=2\alpha_x r \sigma^2_{r} +\sigma^2\Big(  \norm{\vI-\vC}_{C}^{2}\hlambda^2   +  2\norm{\vI-\vR}_{R}^{2} \Big)  \cdot \frac{18}{\alpha_x r \delta}
\ge 2\alpha_x r \sigma^2_{r}+\sigma^2\Big(  \norm{\vI-\vC}_{C}^{2}\hlambda^2  \gamma_{y}^2 +  2\norm{\vI-\vR}_{R}^{2}\gamma_{x}^{2} \Big)  \cdot \frac{18}{\alpha_x r \delta}$. 

\subsubsection{Fourth  inequality}
Recalling the  updates of the variables $\vH_{y}^{k+1}$ and $\vY^{k+1}$ in Algorithm \ref{Alg:RCGTe}, we know
\begin{align}
\nonumber	&\vY^{k+1}-\vH_{y}^{k+1}\\
\nonumber=&\vY^{k+1}-\tY^{k}+\tY^{k}-(\vH^{k}_{y}+\alpha_y\vQ^{k}_{y})\\
\nonumber	=&\vY^{k+1}-\tY^{k}+(1-\alpha_y r)(\tY^{k}-\vH^{k}_{y})\\
\nonumber	& +\alpha_y r(\tY^{k}-\vH^{k}_{y}- \frac{\vQ^{k}_{y}}{r}).
\end{align}
Based on Lemma \ref{lem:UV},  we have
\begin{align}\label{Ykp_Hkyp1}
\nonumber&\EE[\norm{ \vY^{k+1}-\vH_{y}^{k+1}}_{F}^{2} \vert \cF^{k}]\\
\nonumber\leq& (1+\frac{2}{\alpha_y r \delta})\EE[\norm{\vY^{k+1}-\tY^{k} }_{F}^{2}\vert \cF^{k}]\\
&+ (1-\frac{\alpha_y r \delta}{2})\EE[\norm{\tY^{k}-\vH^{k}_{y}}_{F}^{2}\vert \cF^{k}]
+2s_k^{2}\alpha_y r \sigma^2_{r},
\end{align}
where similar technique as in \eqref{tXkp_Hxkp} is used. The following step is to bound $\EE[\norm{\vY^{k+1}-\tY^{k} }_{F}^{2}\vert \cF^{k}]$. Recalling the update of $\vY^{k+1}$ in Algorithm \ref{Alg:RCGTe}, we get
\begin{equation*}
\vY^{k+1}-\tY^{k}=\gamma_{y}(\vI-\vC)(\tY^{k}-\widehat{\vY}^{k})-\gamma_{y}(\vI-\vC)\tY^{k}.
\end{equation*}
Using Lemma \ref{lem:UV}, we have
\begin{align*}
&\EE[\norm{\vY^{k+1}-\tY^{k} }_{F}^{2}\vert \cF^{k}]\\
\nonumber \leq& 2\EE[\norm{\gamma_{y}(\vI-\vC)(\tY^{k}-\widehat{\vY}^{k})}_{F}^{2}\vert \cF^{k}]\\
\nonumber &+2\EE[\norm{ \gamma_{y}(\vI-\vC)\tY^{k} }_{F}^{2}\vert \cF^{k}].
\end{align*}
Reviewing the update $\tY^{k}=\vY^{k}+\nabla\vF(\vX^{k+1})-\nabla\vF(\vX^{k})$, we know
\begin{align*}
&\EE[\norm{\vY^{k+1}-\tY^{k} }_{F}^{2}\vert \cF^{k}]\\
\leq& 2\EE[\norm{\gamma_{y}(\vI-\vC)(\tY^{k}-\widehat{\vY}^{k})}_{F}^{2}\vert \cF^{k}]\\
&+4\EE[\norm{ \gamma_{y}(\vI-\vC)\Pi_{C}\vY^{k}}_{F}^{2}\vert \cF^{k}]\\
&+4\EE[\norm{ \gamma_{y}(\vI-\vC)(\nabla\vF(\vX^{k+1})-\nabla\vF(\vX^{k})) }_{F}^{2}\vert \cF^{k}].
\end{align*}
Using \eqref{tYk_hYk}, we further have
\begin{align}\label{Ykp_tYk1}
\nonumber &\EE[\norm{\vY^{k+1}-\tY^{k} }_{F}^{2}\vert \cF^{k}]\\
\nonumber \leq& 4\gamma_{y}^2\norm{ \vI-\vC}_{C}^{2}C \norm{ \vY^{k}-\vH^{k}_{y} }_{F}^{2}\\
\nonumber &+4\gamma_{y}^2\norm{ \vI-\vC}_{C}^{2}C\EE[\norm{ \nabla\vF(\vX^{k+1})-\nabla\vF(\vX^{k})) }_{F}^{2}\vert \cF^{k}]\\
\nonumber &+2\gamma_{y}^2\norm{ \vI-\vC}_{C}^{2}s_{k}^2\sigma^2\\ 
\nonumber&+4\gamma_{y}^2\norm{ \vI-\vC}_{C}^{2} \norm{ \Pi_{C}\vY^{k} }_{F}^{2} \\
\nonumber &+4\gamma_{y}^2\norm{ \vI-\vC}_{C}^{2} \EE[\norm{  \nabla\vF(\vX^{k+1})-\nabla\vF(\vX^{k}))  }_{F}^{2}\vert \cF^{k}]\\
\nonumber \leq& 4\gamma_{y}^2\norm{ \vI-\vC}_{C}^{2}C \norm{ \vY^{k}-\vH^{k}_{y} }_{F}^{2}
+4\gamma_{y}^2\norm{ \vI-\vC}_{C}^{2} \norm{ \Pi_{C}\vY^{k} }_{F}^{2}\\
\nonumber&+4(C+1)\gamma_{y}^2\norm{ \vI-\vC}_{C}^{2}L^2\EE[\norm{  \vX^{k+1} -\vX^{k}  }_{F}^{2}\vert \cF^{k}]\\
&+2\gamma_{y}^2\norm{ \vI-\vC}_{C}^{2}s_{k}^2\sigma^2.
\end{align}
Besides, we bound $\EE[\norm{\tY^{k}-\vH^{k}_{y}}_{F}^{2}\vert \cF^{k}]$ as follows.
\begin{align}\label{tYk_Hky1}
\nonumber	&\EE[\norm{\tY^{k}-\vH^{k}_{y}}_{F}^{2}\vert \cF^{k}]\\
\nonumber\leq& (1+\frac{\alpha_y r \delta}{4})\norm{\vY^{k}-\vH^{k}_{y}}_{F}^{2}\\
\nonumber &+(1+\frac{4}{\alpha_y r \delta}) \EE[\norm{ \nabla\vF(\vX^{k+1})-\nabla\vF(\vX^{k})) }_{F}^{2}\vert \cF^{k}]\\
\nonumber	\leq& (1+\frac{\alpha_y r \delta}{4})\norm{\vY^{k}-\vH^{k}_{y}}_{F}^{2}\\
&+(1+\frac{4}{\alpha_y r \delta})L^2\EE[\norm{\vX^{k+1}-\vX^{k} }_{F}^{2}\vert \cF^{k}].
\end{align}
Putting \eqref{Ykp_tYk1} and \eqref{tYk_Hky1} back into \eqref{Ykp_Hkyp1}, we get
\begin{align}\label{Ykp_Hkyp2}
\nonumber&\EE[\norm{ \vY^{k+1}-\vH_{y}^{k+1}}_{F}^{2} \vert \cF^{k}]\\
\nonumber \leq& (1+\frac{2}{\alpha_y r \delta})\Big(4\gamma_{y}^2\norm{ \vI-\vC}_{C}^{2}C \norm{ \vY^{k}-\vH^{k}_{y} }_{F}^{2}\\
\nonumber&+2\gamma_{y}^2\norm{ \vI-\vC}_{C}^{2}s_{k}^2\sigma^2
+4\gamma_{y}^2\norm{ \vI-\vC}_{C}^{2} \norm{ \Pi_{C}\vY^{k} }_{F}^{2} \\
\nonumber&+4(C+1)\gamma_{y}^2\norm{ \vI-\vC}_{C}^{2}L^2\EE[\norm{  \vX^{k+1} -\vX^{k}  }_{F}^{2}\vert \cF^{k}]
\Big)\\
\nonumber	&+ (1-\frac{\alpha_y r \delta}{2})\Big[(1+\frac{\alpha_y r \delta}{4})\norm{\vY^{k}-\vH^{k}_{y}}_{F}^{2}\\
\nonumber&+(1+\frac{4}{\alpha_y r \delta})L^2\EE[\norm{\vX^{k+1}-\vX^{k} }_{F}^{2}\vert \cF^{k}]\Big]
+2s_k^{2}\alpha_y r \sigma^2_{r}\\
\nonumber\leq& \left[(1-\frac{\alpha_y r \delta}{4})+ 4C\gamma_{y}^2\norm{\vI-\vC}_{C}^{2} \frac{3}{\alpha_y r \delta}\right]\norm{\vY^{k}-\vH^{k}_{y}}_{F}^{2}\\
\nonumber&+\Big[4(C+1)L^2\gamma_{y}^2\norm{\vI-\vC}_{C}^{2} \frac{3}{\alpha_y r \delta}\\ \nonumber&\qquad+\frac{4}{\alpha_y r \delta}\Big]L^2\EE[\norm{\vX^{k+1}-\vX^{k} }_{F}^{2}\vert \cF^{k}]\\
\nonumber&+4\gamma_{y}^2\norm{\vI-\vC}_{C}^{2}\frac{3}{\alpha_y r \delta}\norm{\Pi_{C}\vY^{k} }_{C}^{2}\\
&+2s_{k}^2\sigma^2\gamma_{y}^2\norm{\vI-\vC}_{C}^{2}\frac{3}{\alpha_y r \delta}
+2s_k^{2}\alpha_y r \sigma^2_{r},
\end{align}
where we use the inequalities $1+\frac{2}{\alpha_y r \delta}\leq  \frac{3}{\alpha_y r \delta}$ and $(1-\frac{\alpha_y r \delta}{2})(1+\frac{4}{\alpha_y r \delta})\leq \frac{4}{\alpha_y r \delta}$. 
Plugging \eqref{Xkp_XkNorm} into \eqref{Ykp_Hkyp2}, we obtain
\begin{align}\label{Ykp_Hkyp3}
\nonumber&\EE[\norm{ \vY^{k+1}-\vH_{y}^{k+1}}_{F}^{2} \vert \cF^{k}]\\
\nonumber\leq& \left[(1-\frac{\alpha_y r \delta}{4})+ 4C\gamma_{y}^2\norm{\vI-\vC}_{C}^{2} \frac{3}{\alpha_y r \delta}\right]\norm{\vY^{k}-\vH^{k}_{y}}_{F}^{2}\\
\nonumber&+\Big[4(C+1)L^2\gamma_{y}^2\norm{\vI-\vC}_{C}^{2} \frac{3}{\alpha_y r \delta}\\
\nonumber&\qquad + \frac{4}{\alpha_y r \delta}\Big]L^2\Bigg(
3\gamma_{x}^{2}\norm{\vI-\vR}_{R}^{2}C\norm{\widetilde{\vX}^{k}-\vH_{x}^{k}}_{F}^{2}\\
\nonumber	&+3\gamma_{x}^{2}\norm{\vI-\vR}_{R}^{2}s_k^{2}\sigma^2
+3\gamma_{x}^{2}\norm{\vI-\vR}_{R}^{2}\norm{\Pi_{R}\vX^{k}}_{F}^{2}\\
\nonumber &+3\hlambda^2\norm{\vR_{\gamma}}_{R}^{2}\norm{ \vY^{k}}_{F}^{2} \Bigg)\\
\nonumber&+4\gamma_{y}^2\norm{\vI-\vC}_{C}^{2} \frac{3}{\alpha_y r \delta}\norm{\Pi_{C}\vY^{k} }_{C}^{2}\\
&+2s_{k}^2\sigma^2\gamma_{y}^2\norm{\vI-\vC}_{C}^{2}\frac{3}{\alpha_y r \delta}
+2s_k^{2}\alpha_y r \sigma^2_{r}.
\end{align}
%Noticing that $1 \leq  \frac{1}{\alpha_y r \delta}$, we have 
Then, we get
\begin{align}
\nonumber&\EE[\norm{ \vY^{k+1}-\vH_{y}^{k+1}}_{F}^{2} \vert \cF^{k}]\\
\nonumber\leq& (1-\frac{\alpha_y r \delta}{4})\norm{\vY^{k}-\vH^{k}_{y}}_{F}^{2}\\
\nonumber&+\frac{1}{\alpha_y r \delta}\Bigg\{
d_{55}C\gamma_{y}^2 \norm{\vY^{k}-\vH^{k}_{y}}_{F}^{2}
+d_{53}\gamma_{y}^2 \norm{\Pi_{C}\vY^{k} }_{C}^{2}\\
\nonumber&+e_2 L^2\bigg[
d_{54}C\gamma_{x}^{2} \norm{\widetilde{\vX}^{k}-\vH_{x}^{k}}_{F}^{2}
+d_{52}\gamma_{x}^{2} \norm{\Pi_{R}\vX^{k}}_{F}^{2}\\
&~+c_5\hlambda^2\norm{ \vY^{k}}_{F}^{2} \bigg]
\Bigg\}	
+s_{k}^2\zeta_{cy},
\label{Ykp_Hkyp4}
\end{align}
where 
$c_5=3\norm{\vR_{\gamma}}_{R}^{2}$, 
$d_3=d_{53}=d_{55}=12\norm{\vI-\vC}_{C}^{2}$, 
$d_4=4$, 
$e_2=d_3(C+1) + d_4 \geq  d_3(C+1)\gamma_{y}^2+ d_4 $, 
$d_{52}=d_{54}=  3 \norm{\vI-\vR}_{R}^{2}$, 
$\zeta_{cy}=\sigma^2\Big[2\norm{\vI-\vC}_{C}^{2}
+e_2L^2\norm{\vI-\vR}_{R}^{2}\Big] \frac{3}{\alpha_y r \delta}
+ 2\alpha_y r \sigma^2_{r}
\ge \sigma^2\Big[2\norm{\vI-\vC}_{C}^{2}\gamma_{y}^2+e_2 L^2\norm{\vI-\vR}_{R}^{2}\gamma_{x}^{2}\Big] \frac{3}{\alpha_y r \delta}+ 
2\alpha_y r \sigma^2_{r}$.

\subsection{Proof of Lemma \ref{Lem:descent}}\label{Pf:descent}
From Assumption \ref{Assumption: function}, the gradient of $f$ is $L$-Lipschitz continuous. In addition, we know $\oX^{k+1}=\oX^{k}-\frac{1}{n}\vu_{R}^{\T}\Lambda \vY^{k}$. Then, we have
\begin{align}
\nonumber &f(\oX^{k+1})\\
\nonumber \leq& f(\oX^{k})+\langle \nabla f(\oX^{k}), \oX^{k+1}-\oX^{k}\rangle 
+\frac{L}{2}\norm{\oX^{k+1}-\oX^{k}}^2\\
%\nonumber =&f(\oX^{k})+\langle \nabla f(\oX^{k}), -\frac{1}{n}\vu_{R}^{\T}\Lambda \vY^{k}\rangle +\frac{L}{2}\norm{-\frac{1}{n}\vu_{R}^{\T}\Lambda \vY^{k}}^2\\
\nonumber =&f(\oX^{k})-\overline{\lambda}\langle \nabla f(\oX^{k}),  \frac{1}{n\overline{\lambda}}\vu_{R}^{\T}\Lambda \vY^{k}\rangle +\frac{L\overline{\lambda}^2}{2}\norm{\frac{1}{n\overline{\lambda}}\vu_{R}^{\T}\Lambda \vY^{k}}^2\\
\nonumber =&f(\oX^{k})-\frac{\overline{\lambda}}{2} \Bigg( \norm{\nabla f(\oX^{k})}^2  +\norm{\frac{1}{n\overline{\lambda}}\vu_{R}^{\T}\Lambda \vY^{k}}^2 \\
\nonumber &-\norm{\nabla f(\oX^{k})-\frac{1}{n\overline{\lambda}}\vu_{R}^{\T}\Lambda \vY^{k} }^2 \Bigg) 
\nonumber+\frac{L\overline{\lambda}^2}{2}\norm{\frac{1}{n\overline{\lambda}}\vu_{R}^{\T}\Lambda \vY^{k}}^2,
%\nonumber =&f(\oX^{k})-\frac{\overline{\lambda}}{2} \Bigg( \norm{\nabla f(\oX^{k})}^2  +\norm{\frac{1}{n\overline{\lambda}}\vu_{R}^{\T}\Lambda \vY^{k}}^2 \\
%\nonumber &-\norm{\nabla f(\oX^{k})-\oY^{k} -\frac{1}{n\overline{\lambda}}\vu_{R}^{\T}\Lambda \Pi_{C}\vY^{k} }^2 \Bigg) \\
%\nonumber &+\frac{L\overline{\lambda}^2}{2}\norm{\frac{1}{n\overline{\lambda}}\vu_{R}^{\T}\Lambda \vY^{k}}^2,
\end{align}
where the second equality is based on $\langle \va,\vb \rangle=\frac{1}{2}(\norm{\va}^2+\norm{\vb}^2-\norm{\va-\vb}^2)$. 
Note that $\overline{\lambda}\leq\frac{1}{L}$ from the assumption in Lemma \ref{Lem:descent}, we know 
$\frac{L\overline{\lambda}^2}{2}\norm{\frac{1}{n\overline{\lambda}}\vu_{R}^{\T}\Lambda \vY^{k}}^2
\leq  \frac{\overline{\lambda}}{2}\norm{\frac{1}{n\overline{\lambda}}\vu_{R}^{\T}\Lambda \vY^{k}}^2$. Meanwhile, $\nabla f(\oX^{k})-\frac{1}{n\overline{\lambda}}\vu_{R}^{\T}\Lambda \vY^{k}=\nabla f(\oX^{k})-\oY^{k} -\frac{1}{n\overline{\lambda}}\vu_{R}^{\T}\Lambda \Pi_{C}\vY^{k}$. Thus, we have
\begin{align}%\label{fplus}
\nonumber  &f(\oX^{k+1})\\
\nonumber\leq &f(\oX^{k})-\frac{\overline{\lambda}}{2}\norm{\nabla f(\oX^{k})}^2  \\  
\nonumber &+\frac{\overline{\lambda}}{2}\norm{\nabla f(\oX^{k})-\oY^{k} -\frac{1}{n\overline{\lambda}}\vu_{R}^{\T}\Lambda \Pi_{C}\vY^{k} }^2 \\
\nonumber \leq&f(\oX^{k})-\frac{\overline{\lambda}}{2}   \norm{\nabla f(\oX^{k})}^2   +\frac{L^2\overline{\lambda}}{n}\norm{\Pi_{R}\vX^{k}}_{R}^{2}\\
\nonumber &+\frac{\norm{\vu_{R}}^{2}\hlambda^2}{n^2\overline{\lambda}}\norm{\Pi_{C}\vY^{k}}_{C}^{2}\\
\nonumber \leq&f(\oX^{k})-\frac{M\hlambda}{2}   \norm{\nabla f(\oX^{k})}^2  \\ 
\nonumber&+\frac{\norm{\vu_{R}} \norm{\vu_{C}}}{n^2} \hlambda L^2\norm{\Pi_{R}\vX^{k}}_{R}^{2}
+\frac{\norm{\vu_{R}}^{2}}{n^2M}\hlambda\norm{\Pi_{C}\vY^{k}}_{C}^{2}, 
\end{align}
where we use Lemma \ref{lem:UV} in the  second inequality,
%$\overline{\lambda}\leq\frac{1}{2L}$ for $\frac{L\overline{\lambda}^2}{2}\norm{\frac{1}{n\overline{\lambda}}\vu_{R}^{\T}\Lambda \vY^{k}}_{2}^2$, 
the assumption $\overline{\lambda}\geq M \hlambda$ and the definition $\overline{\lambda}=\frac{1}{n}\vu_{R}^{\T}\Lambda \vu_{C}\leq \frac{1}{n}\vu_{R}^{\T}\vu_{C}\hlambda\leq \frac{1}{n}\norm{\vu_{R}}\norm{\vu_{C}}\hlambda
$ in the third inequality.

\subsection{Proof of Theorem \ref{Thm:RCPP}}\label{Pf:RCPP}
%Combining Lemmas \textcolor{red}{XXX} and \textcolor{red}{XXX}, we get
Let $V^{k}=L^2\EE[\norm{\Pi_{R}\vX^{k}}_{R}^{2}]+A\EE[\norm{\Pi_{C}\vY^{k}}_{C}^{2}]+B\EE[\norm{\tX^{k}-\vH^{k}_{x}}_{F}^{2}]+D\EE[\norm{ \vY^{k}-\vH_{y}^{k}}_{F}^{2} ]$,
where 
$A=\frac{\theta_{C}\gamma_{y}\theta_{R}}{4 e_1 c_{32}\gamma_{x}}$, 
$B=\frac{L^2 \alpha_x r \delta \theta_{R}}{4 d_{42}\gamma_x}$,
$D=\frac{\theta_{C}\gamma_{y} \alpha_y r \delta \theta_{R}}{4 e_2 d_{52}\gamma_x}\leq\frac{  \alpha_y r \delta \theta_{R}}{4 e_2 d_{52}\gamma_x}$. 
Combining Lemmas \ref{Lem:Yk} and \ref{Lem:RCPP} with the conditions on $\gamma_{x},\gamma_{y},\hlambda$, we have
\begin{equation*}
\begin{aligned}
&V^{k+1}\\
\leq& \left( 1-\frac{7\theta_{R}\gamma_x}{32}
\right) L^2 \norm{\Pi_{R}\vX^{k}}_{R}^{2}
+\left( 1-\frac{\theta_{C}\gamma_y }{4} 
\right) A\norm{\Pi_{C}\vY^{k}}_{C}^{2}\\
&+\left(1-\frac{\alpha_x r \delta}{4}
\right)B\norm{\tX^{k}-\vH^{k}_{x}}_{F}^{2}\\
&+\left(1-\frac{\alpha_y r \delta}{16}
\right)D\norm{\vY^{k}-\vH^{k}_{y}}_{F}^{2}\\
%&+\left(
%\frac{\theta_{R}\gamma_x}{8}
%+\frac{\theta_{C}\gamma_y }{8}
%+\frac{\alpha_x r \delta}{8}
%+\frac{\alpha_y r \delta}{16}
%\right)
&+ \frac{\beta}{4 E}\norm{\nabla f(\oX^{k})}^2
+s_k^{2}\zeta_0,
\end{aligned}
\end{equation*}
where $\beta=\frac{\theta_{R}\gamma_{x}}{8}$ and $E=\frac{\norm{\vu_{R}}\norm{\vu_{C}} }{n^2M}$.
Define $U^{k}=f(\oX^{k})-f(\vX^*)$. 
Combining Lemma \ref{Lem:descent}, we obtain
%Combining with relation \eqref{descent}, we obtain
%where we use $s_k^{2}=c_0c^k$
\begin{align*}
&U^{k+1}+\frac{E M\hlambda}{\beta} V^{k+1}\\
%\leq& U^{k}-\frac{M\hlambda}{2}   \norm{\nabla f(\oX^{k})}^2\\
%&+E_1M\hlambda  L^2\norm{\Pi_{R}\vX^{k}}_{R}^{2}
%+E_2M\hlambda\norm{\Pi_{C}\vY^{k}}_{C}^{2}\\
%&+\frac{E M\hlambda}{\beta}\Bigg[\left( 1-\frac{7\theta_{R}\gamma_x}{32}
%\right) L^2 \norm{\Pi_{R}\vX^{k}}_{R}^{2}\\
%&+\left( 1-\frac{\theta_{C}\gamma_y }{4} 
%\right) A\norm{\Pi_{C}\vY^{k}}_{C}^{2}
%+\left(1-\frac{\alpha_x r \delta}{4}
%\right)B\norm{\tX^{k}-\vH^{k}_{x}}_{F}^{2}\\
%&+\left(1-\frac{\alpha_y r \delta}{16}
%\right)D\norm{\vY^{k}-\vH^{k}_{y}}_{F}^{2}
%+ \frac{\beta}{4 E}\norm{\nabla f(\oX^{k})}^2
%+s_k^{2}\zeta_0\Bigg]\\
%\leq& U^{k}-\frac{M\hlambda}{4}   \norm{\nabla f(\oX^{k})}^2
%+\frac{\theta_{R}\gamma_{x}}{8} \frac{EM\hlambda}{\beta}  L^2\norm{\Pi_{R}\vX^{k}}_{R}^{2}\\
%&+ \frac{\theta_{R}\gamma_{x}}{8}\frac{8}{\theta_{C}\gamma_{y}}\frac{E_2}{ E A}\frac{\theta_{C}\gamma_{y}}{8}\frac{EM\hlambda}{\beta}A\norm{\Pi_{C}\vY^{k}}_{C}^{2}\\
%&+\frac{E M\hlambda}{\beta}\Bigg[\left( 1-\frac{7\theta_{R}\gamma_x}{32}
%\right) L^2 \norm{\Pi_{R}\vX^{k}}_{R}^{2}\\
%&+\left( 1-\frac{\theta_{C}\gamma_y }{4} 
%\right) A\norm{\Pi_{C}\vY^{k}}_{C}^{2}
%+\left(1-\frac{\alpha_x r \delta}{4}
%\right)B\norm{\tX^{k}-\vH^{k}_{x}}_{F}^{2}\\
%&+\left(1-\frac{\alpha_y r \delta}{16}
%\right)D\norm{\vY^{k}-\vH^{k}_{y}}_{F}^{2}
%+s_k^{2}\zeta_0,\Bigg]\\
\leq& U^{k}-\frac{M\hlambda}{4}   \norm{\nabla f(\oX^{k})}^2
+\frac{E M\hlambda}{\beta}\Bigg[\left( 1-\frac{3\theta_{R}\gamma_x}{32}
\right) L^2 \norm{\Pi_{R}\vX^{k}}_{R}^{2}\\
&+\left( 1-\frac{\theta_{C}\gamma_y }{8} 
\right) A\norm{\Pi_{C}\vY^{k}}_{C}^{2}
+\left(1-\frac{\alpha_x r \delta}{4}
\right)B\norm{\tX^{k}-\vH^{k}_{x}}_{F}^{2}\\
&+\left(1-\frac{\alpha_y r \delta}{16}
\right)D\norm{\vY^{k}-\vH^{k}_{y}}_{F}^{2}
\Bigg]+s_k^{2}\tilde{\zeta_0},
%\leq&(1-\frac{1}{4}M\hlambda)U^{k}+(1-\beta) \frac{E M\hlambda}{\beta} V^{k}
% +\beta\frac{\norm{\vu_{R}}\norm{\vu_{C}} }{2nME} \frac{E M\hlambda}{\beta}L^2\norm{\Pi_{R}\vX^{k}}_{R}^{2}
%+\beta\frac{\norm{\vu_{R}}^{2}}{2n^2M^2 E A} \frac{E M\hlambda}{\beta}A\norm{\Pi_{C}\vY^{k}}_{C}^{2}+s_k^{2}\tilde{\zeta_0},
\end{align*}
where $\tilde{\zeta_0}=\zeta_0\frac{E M\hlambda}{\beta}$ and we use the fact $\gamma_{x}\leq \frac{\sqrt{M\norm{\vu_{C}}}}{2\sqrt{\norm{\vu_{R}}e_1 c_{32}}}\theta_{C}\gamma_y$. 
Denote $\tilde{\rho}=\max\{1-\frac{1}{2}M\hlambda\mu,1-\frac{\theta_{R}\gamma_x}{16},1-\frac{\theta_{C}\gamma_y}{8},1-\frac{\alpha_x r \delta}{4},1-\frac{\alpha_y r \delta}{16}\}$ and choose $s_k^2=c_0c^k$ where $c\in(\tilde{\rho},1)$. Recalling the PL condition, we obtain $-\frac{M\hlambda}{4}\norm{\nabla f(\oX^{k})}^2\leq -\frac{M\hlambda}{2} U^{k} $. Thus, we have
\begin{align*}
&U^{k+1}+\frac{E M\hlambda}{\beta} V^{k+1}\\
%\leq&(1-\frac{1}{4}M\hlambda)U^{k}+(1-\frac{\beta}{2}) \frac{E M\hlambda}{\beta} V^{k}+s_k^{2}\tilde{\zeta_0}\\
\leq&\tilde{\rho}(U^{k} + \frac{E M\hlambda}{\beta} V^{k})+s_k^{2}\tilde{\zeta_0}\\
\leq&\tilde{\rho}^{k+1}(U^{0} + \frac{E M\hlambda}{\beta} V^{0})+\sum_{l=0}^{k}\tilde{\rho}^{k-l}c^{l}\Theta\\
\leq&\tilde{\rho}^{k+1}(U^{0} + \frac{E M\hlambda}{\beta} V^{0})+ c^{k}\Theta\sum_{l=0}^{k}\left(\frac{\tilde{\rho}}{c}\right)^{k-l}\\
\leq&\tilde{\rho}^{k+1}(U^{0} + \frac{E M\hlambda}{\beta} V^{0})+c^{k+1} \frac{\Theta}{c-\tilde{\rho}},
\end{align*}
where $\Theta=c_0\tilde{\zeta_0}$.

% or

%			\appendices
%			\section{Proofs for XXX}	
%			\subsection{Proof of Lemma X1}
%			\label{subsec: proof eigenvectors u v}
%		
%			
%			\section{Proofs for YYY}
%			\subsection{Proof of Lemma Y1}
%			\label{proof: lem: x-xstar_pre gpp}

%\section*{Acknowledgment}

\ifCLASSOPTIONcaptionsoff
\newpage
\fi

\end{document}